\documentclass[11pt, a4paper]{amsart}
\usepackage{a4, a4wide}
\usepackage{amssymb}
\usepackage{amsmath}
\usepackage{amsthm}
\usepackage{amstext}
\usepackage{amscd}
\usepackage{latexsym}
\usepackage{graphics}
\usepackage{color}
\usepackage{enumitem}
\usepackage{mathabx}
\usepackage[all]{xy}

\allowdisplaybreaks

\DeclareSymbolFont{symbolsC}{U}{ntxsyc}{m}{n}
\SetSymbolFont{symbolsC}{bold}{U}{ntxsyc}{b}{n}
\DeclareMathSymbol{\nsimeq}{\mathrel}{symbolsC}{59}

\usepackage[colorlinks, pagebackref]{hyperref}
\hypersetup{
  colorlinks=true,
  citecolor=blue,
  linkcolor=blue,
  urlcolor=blue}
  
\makeatletter
\def\@tocline#1#2#3#4#5#6#7{\relax
  \ifnum #1>\c@tocdepth % then omit
  \else
    \par \addpenalty\@secpenalty\addvspace{#2}%
    \begingroup \hyphenpenalty\@M
    \@ifempty{#4}{%
      \@tempdima\csname r@tocindent\number#1\endcsname\relax
    }{%
      \@tempdima#4\relax
    }%
    \parindent\z@ \leftskip#3\relax \advance\leftskip\@tempdima\relax
    \rightskip\@pnumwidth plus4em \parfillskip-\@pnumwidth
    #5\leavevmode\hskip-\@tempdima
      \ifcase #1
      \or\or \hskip 2em \or \hskip 2homologyem \else \hskip 3em \fi%
      #6\nobreak\relax
    \dotfill\hbox to\@pnumwidth{\@tocpagenum{#7}}\par
    \nobreak
    \endgroup
  \fi}
\makeatother

\theoremstyle{plain}

\newtheorem{theorem}{Theorem}[section]
\newtheorem{question}[theorem]{Question}
\newtheorem{lemma}[theorem]{Lemma}

\newtheorem{proposition}[theorem]{Proposition}
\newtheorem{conjecture}[theorem]{Conjecture}
\theoremstyle{definition}

\newtheorem{conventions}[theorem]{Conventions}
\newtheorem{notation}[theorem]{Notation}

\newtheorem{remark}[theorem]{Remark}

\newtheorem{definition}[theorem]{Definition}
\newtheorem{example}[theorem]{Example}

\numberwithin{equation}{section}
\newcommand{\ilim}{\mathop{\varprojlim}\limits} % inverse limit
\newcommand{\dlim}{\mathop{\varinjlim}\limits}  % direct limit

\newcommand{\Supp}{{\rm Supp}}

\newcommand{\Proj}{{\rm Proj \,}}

\newcommand{\Hom}{{\rm Hom}}

\newcommand{\Spec}{{\rm Spec \,}}

\renewcommand{\tilde}{\widetilde}
\newcommand{\sF}{{\mathcal F}}

\newcommand{\sL}{{\mathcal L}}

\newcommand{\sS}{{\mathcal S}}

\newcommand{\A}{{\mathbb A}}

\newcommand{\C}{{\mathbb C}}

\renewcommand{\P}{{\mathbb P}}

\newcommand{\Sing}{{\rm Sing_*^{\A^1}}} 
\newcommand{\Sp}{{\rm Sp}} 		

\def\<{\langle}
\def\>{\rangle} 
\def\-{\overline} 
\def\~{\widetilde}
\def\^{\widehat}
\def\fr{\mathfrak}
\def\@{\mathcal}
\def\!{\mathscr}
\def\#{\mathbb}
\def\_{\underline}

\input{xy}
\xyoption{all}
\begin{document}

\title{$\A^1$-connected components of ruled surfaces}

\author{Chetan Balwe}
\address{Department of Mathematical Sciences, Indian Institute of Science Education and Research Mohali, Knowledge City, Sector-81, Mohali 140306, India.}
\email{cbalwe@iisermohali.ac.in}

\author{Anand Sawant}
\address{School of Mathematics, Tata Institute of Fundamental Research, Homi Bhabha Road, Colaba, Mumbai 400005, India.}
\email{asawant@math.tifr.res.in}
%\date{\today}
\date{}
\thanks{Chetan Balwe was supported by SERB-DST MATRICS Grant: MTR/2017/000690}
\thanks{Anand Sawant was supported by DFG Individual Research Grant SA3350/1-1 while he was affiliated to Ludwig-Maximilians-Universit\"at, Munich, where a part of this work was carried out. He acknowledges support of the Department of Atomic Energy, Government of India, under project no. 12-R\&D-TFR-5.01-0500.}
%\subjclass[2010]{14C15, 14C25, 19E15 (Primary)}
%\keywords{algebraic cycles; motivic cohomology; rost nilpotence}

\begin{abstract}
A conjecture of Morel asserts that the sheaf of $\A^1$-connected components of a space is $\A^1$-invariant.  Using purely algebro-geometric methods, we determine the sheaf of $\A^1$-connected components of a smooth projective surface, which is birationally ruled over a curve of genus $>0$.  As a consequence, we show that Morel's conjecture holds for all smooth projective surfaces over an algebraically closed field of characteristic $0$.
\end{abstract}

\maketitle
\tableofcontents

\setlength{\parskip}{2.5pt plus2pt minus1pt}
%\raggedbottom

\section{Introduction}
\label{section introduction}

Since its inception in the foundational work of Morel and Voevodsky \cite{Morel-Voevodsky}, $\A^1$-homotopy theory has provided a systematic framework to successfully adapt several techniques of algebraic topology to the realm of algebraic geometry by having the affine line $\A^1$ play the role of the unit interval.  Fix a base field $k$ and let $Sm/k$ denote the category of smooth, finite-type, separated schemes over $k$.  Let $\mathcal H(k)$ denote the $\mathbb A^1$-homotopy category, which is obtained by taking a suitable localization of the category of simplicial sheaves of sets (also called \emph{spaces}) on $Sm/k$ for the Nisnevich topology.  In $\A^1$-algebraic topology, one then studies the $\mathbb A^1$-homotopy sheaves of a (pointed) space.  In topology, the set of connected components of a topological space and the homotopy groups of a (pointed) topological space are discrete as topological spaces.  The counterpart of discreteness in $\A^1$-homotopy theory is the notion of \emph{$\A^1$-invariance}.

It is thus very natural to ask if the $\mathbb A^1$-homotopy sheaves of a (pointed) space $\mathcal X$ are $\mathbb A^1$-invariant.  Morel \cite[Theorem 6.1, Corollary 6.2]{Morel} has shown that the homotopy sheaves $\pi_n^{\mathbb A^1}(\mathcal X, x)$, for $n\geq 1$, are $\mathbb A^1$-invariant.  However, quite incredibly, $\mathbb A^1$-invariance of the sheaf of $\mathbb A^1$-connected components is not yet known and has been conjectured by Morel \cite[Conjecture 1.12]{Morel}.  

\begin{conjecture}[Morel]
\label{conjecture Morel}
For any space $\@X$, the sheaf $\pi_0^{\A^1}(\@X)$ is $\A^1$-invariant.
\end{conjecture}

Let $\@X$ be a space.  A crucial drawback one faces while handling $\A^1$-connected components as opposed to higher homotopy sheaves of $\@X$ is the lack of a natural group structure on $\pi_0^{\A^1}(\@X)$ in general.  Another serious difficulty in handling $\pi_0^{\A^1}(\@X)$ is presented by the fact that the explicit description of the \emph{$\A^1$-fibrant replacement functor} given in \cite[p. 107]{Morel-Voevodsky} is uncomputable in general.  Apart from some trivial examples (curves, $\A^1$-rigid spaces, etc.) in which Conjecture \ref{conjecture Morel} clearly holds, it has been verified for $H$-groups and homogeneous spaces for $H$-groups \cite{Choudhury} and projective non-uniruled surfaces \cite{Balwe-Hogadi-Sawant}.  Conjecture \ref{conjecture Morel} is wide open in general, even for smooth projective schemes over an algebraically closed field.  The main result of this paper is as follows.

\begin{theorem}
\label{main theorem}
Let $X$ be a smooth projective birationally ruled surface over an algebraically closed field $k$ of characteristic $0$.  Then the sheaf $\pi_0^{\A^1}(X)$ is $\A^1$-invariant.
\end{theorem}

The case of non-uniruled surfaces (over an arbitrary field) was already handled in \cite{Balwe-Hogadi-Sawant}.  Thus, Conjecture \ref{conjecture Morel} holds for all smooth projective surfaces over an algebraically closed field of characteristic $0$.  Appropriate analogues of some of the ideas used in the proof of Theorem \ref{main theorem} go through while working over an arbitrary field using the classification of surfaces (see \cite[Chapter III, Theorem 2.2]{Kollar}).  We intend to take up the problem of proving Morel's conjecture for surfaces over an arbitrary field in future work.

Since $\pi_0^{\A^1}(X)$ is the Nisnevich sheaf associated with the presheaf $$U \mapsto \Hom_{\@H(k)}(U, X)$$ on $Sm/k$, the study of $\pi_0^{\A^1}(X)$ is closely related to the following fundamental question. 

\begin{question}
\label{question naive homotopies}
How far is $\Hom_{\@H(k)}(U, X)$ from the set equivalence classes of morphisms of schemes $U \to X$ under the equivalence relation generated by naive $\A^1$-homotopies?
\end{question}

One might hope that the two sets in Question \ref{question naive homotopies} agree, at least when $U$ is a smooth henselian local scheme over $k$.  Let $\@S(X)$ denote the sheaf of \emph{$\A^1$-chain homotopy classes} of $X$.  One can iterate the construction of taking $\A^1$-chain connected components and take the limit to form the \emph{universal $\A^1$-invariant quotient} $$\@L(X): = \varinjlim_n \@S^n(X)$$ of $X$.  In view of \cite[Theorem 2.13, Remark 2.15, Corollary
2.18]{Balwe-Hogadi-Sawant}, if Conjecture \ref{conjecture Morel} holds, then the canonical map $\pi_0^{\A^1}(X) \to \sL(X)$ is an isomorphism.  The sheaf $\sL(X)$ thus gives a purely geometric way to study $\pi_0^{\A^1}(X)$ and it is natural to ask the following question.

\begin{question}
\label{main question}
Let $X$ be a scheme over $k$.  Does there exist $n \in \mathbb N$ such that $\@L(X) \simeq \mathcal S^n(X)$?
\end{question}

It is worthwhile to point out that there exist smooth projective schemes $X$ for which neither $\@L(X)$ nor $\pi_0^{A^1}(X)$ agrees with $\@S(X)$.  The first example of a smooth projective scheme $X$ (over $\C$) for which $\@S(X) \nsimeq \@L(X)$ was constructed in \cite[Section 4]{Balwe-Hogadi-Sawant}.  The first general class of smooth varieties $X$ for which $\pi_0^{\A^1}(X)$ does not agree with $\@S(X)$ but with a further iteration of the functor $\@S$ on $X$ is provided by $\A^1$-connected anisotropic groups over an infinite perfect field by the results of \cite{Balwe-Sawant-IMRN} (see \cite[Corollary 3.4]{Sawant-IC}).  Theorem \ref{main theorem} is a consequence of the following general result, which provides the first general family of smooth projective varieties for which $\pi_0^{\A^1}(X)$ is given by a nontrivial iteration of $\@S$ on $X$.

\begin{theorem}
\label{theorem ruled iterations of S}
Let $E$ be a $\mathbb P^1$-bundle over a smooth projective curve $C$ of genus $g>0$ over an algebraically closed field $k$ of characteristic $0$. Let $X$ be a smooth projective surface, which is not isomorphic to $E$ and has $E$ for its minimal model. Then we have the following:
\begin{enumerate}[label=$(\alph*)$]
\item $\pi_0^{\A^1}(E) \simeq \mathcal S^n(E) \simeq C$, for all $n \geq 1$;
\item $\pi_0^{\A^1}(X) \simeq \mathcal S^n(X)$, for all $n \geq 2$.  Moreover, $\pi_0^{\A^1}(X) \nsimeq C$.
\end{enumerate}
%Consequently, $\Sing X$ is not $\A^1$-local.
\end{theorem}

It was expected in \cite[Remark 1.13]{Morel} that for a smooth projective surface $X$ birationally ruled over a curve $C$ of genus $g>0$, one would have $\pi_0^{\A^1}(X) \simeq C$.  Indeed, by results of \cite{Asok-Morel} or \cite{Balwe-Hogadi-Sawant}, it follows that $\pi_0^{\A^1}(X)(\Spec F) \simeq C(\Spec F)$, for every finitely generated, separable field extension $F$ of $k$.  However, Theorem \ref{theorem ruled iterations of S} shows that the situation in reality is quite delicate. 

We now briefly outline the contents of the article.  The central tool of the paper is the notion of \emph{$\A^1$-ghost homotopies}, which need not be morphisms of $\A^1$ into the scheme in question but are defined only over a Nisnevich cover of $\A^1$.  The formalism of $\A^1$-$n$-ghost homotopies, which gives a systematic way to analyze $\A^1$-homotopy classes of sections of the sheaves $\sS^n(X)$ over a smooth henselian local scheme, is described in Section \ref{subsection ghost homotopies}.  In Section \ref{section homotopies on a blowup}, we prove a key rigidity result (Proposition \ref{proposition local generator ghost homotopy}) for sections of a variety admitting a  morphism to an $\A^1$-rigid scheme over a smooth henselian local scheme.  Local analysis of $\A^1$-homotopy classes on ruled surfaces is crucial to our arguments.  We collect the tools needed for the same in Sections \ref{subsection constructions} and \ref{subsection nodal blowups of ruled surfaces}.  In Section \ref{section nodal blowups}, we prove Theorem \ref{theorem ruled iterations of S} for a special class of blowups of $\P^1 \times C$ for a smooth projective $\A^1$-rigid curve $C$, which we refer to as \emph{nodal blowups}.  The trickiest part of the proof here is the explicit construction of certain ghost homotopies (see the proof of Theorem \ref{theorem nodal case}).  In Section \ref{section general case}, we handle the case of a general blowup of $\P^1 \times C$.  This will be accomplished by a key reduction argument using a geometric result presented in Section \ref{subsection etale cover}, which allows us to prove the result by induction on the number of blowups required.  We finally put all these results together, and using some known results on rational surfaces, we obtain a proof of Theorem \ref{main theorem}.

\subsection*{Conventions and notation}

Throughout the paper, $k$ will denote a fixed base field and $Sm/k$ will denote the category of smooth, finite-type, separated schemes over $\Spec k$.  Section \ref{subsection nodal blowups of ruled surfaces} onwards, we will assume that $k$ is algebraically closed of characteristic $0$.  We will denote by $\@H(k)$ the $\A^1$-homotopy category constructed by Morel and Voevodsky \cite{Morel-Voevodsky}, the objects of which are simplicial Nisnevich sheaves of sets on $Sm/k$ (also called \emph{spaces}).  Regarding $\A^1$-connected components and related notions, we follow the conventions and notation of \cite{Balwe-Hogadi-Sawant}.

We will often make use of essentially smooth schemes, that is, schemes which are filtered inverse limits of diagrams of smooth schemes in which the transition maps are \'etale, affine morphisms.  All presheaves on $Sm/k$ will be extended to essentially smooth schemes by defining $\@F(\ilim U_{\alpha}) = \dlim \@F(U_\alpha)$.  By a \emph{smooth henselian local scheme} over $k$, we will mean the henselization of the local ring at a smooth point of a scheme over $k$.

A Nisnevich sheaf of sets $\@F$ on $Sm/k$ is \emph{$\A^1$-invariant} if for every $U \in Sm/k$, the projection map $U \times \A^1 \to U$ induces a bijection $\@F(U) \stackrel{\simeq}{\to} \@F(U \times \A^1)$.  A scheme $X$ over $k$ is called \emph{$\A^1$-rigid} if the associated Nisnevich sheaf given by its functor of points is $\A^1$-invariant.

\section{\texorpdfstring{$\A^1$}{A1}-connected components of schemes}

\subsection{Preliminaries on \texorpdfstring{$\A^1$}{A1}-connected components}
\label{subsection preliminaries}

In this section, we review some preliminary material on $\#A^1$-homotopy theory relevant to the contents of this paper.  We will follow the notation and terminology of \cite{Balwe-Hogadi-Sawant}.

\begin{definition}
\label{definition pi0}
The sheaf of \emph{$\A^1$-connected components} of a space $\@X$, denoted by $\pi_0^{\#A^1}(\@X)$, is defined to be the Nisnevich sheafification of the presheaf $$U \mapsto \Hom_{\@H(k)}(U, \@X)$$ on $Sm/k$.
\end{definition}

When $\@X$ is represented by a scheme $X$, we would like to use geometric properties of $X$ to compute $\pi_0^{\#A^1}(X)$.  A systematic way to go about this is to consider the functor of \emph{$\#A^1$-chain connected components} of $X$, which attempts to capture the notion of $\#A^1$-connectivity in a naive manner.

\begin{notation}
For any scheme $U$ over $k$, we let $\sigma_0$ and $\sigma_1$ denote the morphisms $U \to U \times \#A^1$ given by $u \mapsto (u,0)$ and $u \mapsto (u,1)$, respectively. 
\end{notation}

\begin{definition}
\label{definition homotopy}
Let $\@F$ be a sheaf of sets over $Sm/k$ and let $U$ be an essentially smooth scheme over $k$.
\begin{itemize}
\item[(1)] An \emph{$\#A^1$-homotopy} of $U$ in $\@F$ is an element $h$ of $\@F(U \times \#A^1_{k})$. We say that $t_1, t_2 \in \@F(U)$ are \emph{$\#A^1$-homotopic} if there exists an $\#A^1$-homotopy $h \in \@F(U \times \#A^1)$ such that $\sigma_0^*(h) = t_1$ and $\sigma_1^*(h)= t_2$.
\item[(2)] An \emph{$\#A^1$-chain homotopy} of $U$ in $\@F$ is a finite sequence $h=(h_1, \ldots, h_r)$ where each $h_i$ is an $\#A^1$-homotopy of $U$ in $\@F$ such that $\sigma_1^*(h_i) = \sigma_0^*(h_{i+1})$ for $1 \leq i \leq r-1$. We say that $t_1, t_2 \in \@F(U)$ are \emph{$\#A^1$-chain homotopic} if there exists an $\#A^1$-chain homotopy $h=(h_1, \ldots, h_r)$ such that $\sigma_0^*(h_1) = t_1$ and $\sigma_1^*(h_r) = t_2$.
\end{itemize} 
\end{definition}

Note that for any scheme $U$ and sheaf $\@F$, $\#A^1$-chain homotopy gives an equivalence relation on the set $\@F(U)$. 

\begin{definition} 
\label{definition A1 chain connected components}
Let $\@F$ be a Nisnevich sheaf of sets on $Sm/k$. Define $\@S(\@F)$ to be the Nisnevich sheaf associated with the presheaf 
\[
U \mapsto \@F(U)/\sim
\] 
\noindent for $U \in Sm/k$, where $\sim$ denotes the equivalence relation defined by $\#A^1$-chain homotopy.  In other words, $\@S(\@F)$ is the Nisnevich sheafification of the presheaf $U \mapsto \pi_0(\Sing \@F)$ on $Sm/k$, where $\Sing \@F$ denotes the \emph{Morel-Voevodsky singular construction} on $\@F$ (see \cite[p. 87-88]{Morel-Voevodsky}). 
\end{definition}

Iterating this construction, we obtain a sequence of epimorphisms
\[
\mathcal F \to \mathcal S(\mathcal F) \to \mathcal S^2(\mathcal F) \to \cdots \to \mathcal S^n(\mathcal F) \to \cdots,
\]
where $\mathcal S^{n+1}(\mathcal F)$ is defined inductively to be $\mathcal S(\mathcal S^{n}(\mathcal F))$, for every $n \in \mathbb N$.  The utility of this concept lies in the following result which was proved in \cite[Theorem 2.13, Remark 2.15, Corollary 2.18]{Balwe-Hogadi-Sawant}.

\begin{theorem}
\label{theorem BHS iterations of S}
For any sheaf of sets $\@F$ on $Sm/k$, the sheaf $$\@L(\@F) := \underset{n}{\varinjlim}~ \@S^n(\@F)$$ is $\#A^1$-invariant. Thus, the quotient morphism $\@F \to \@L(\@F)$ uniquely factors through the canonical morphism $\@F \to \pi_0^{\#A^1}(\@F)$. The morphism $\pi_0^{\#A^1}(\@F) \to \@L(\@F)$ is an isomorphism if and only if $\pi_0^{\#A^1}(\@F)$ is $\#A^1$-invariant. 
\end{theorem}

Thus, in order to prove the $\#A^1$-invariance of $\pi_0^{\#A^1}(X)$ for a scheme $X$, one can attempt to compute iterations $\@S^n(X)$ for all $n$ and see if $\pi_0^{\#A^1}(X)$ agrees with their colimit as $n \to \infty$.  For this, we would like to study $\#A^1$-homotopies in the sheaves $\@S^n(X)$, $n \geq 0$. These sheaves are not schemes in general and so if we wish to exploit geometric information, we will have to ``lift'' those homotopies to $X$.  We take up the task of setting up the notation for doing the same in the next subsection. 

\subsection{Ghost homotopies}
\label{subsection ghost homotopies}

Let $\@F$ be a Nisnevich sheaf of sets on $Sm/k$.  The morphism $\@F \to \@S^n(\@F)$ is a surjection of sheaves. Thus, although not every morphism $U \to \@S(\@F)$ can be lifted to $\@F$, there exists some Nisnevich cover $V \to U$ such that the morphism $V \to U \to \@S^n(X)$ lifts to $\@F$. On $V \times_U V$, these lifts must be compatible up to $\#A^1$-chain homotopy.  In other words, the two induced morphisms of $V \times_U V$ are $\#A^1$-chain homotopic (at least after pulling back to some suitable Nisnevich cover of  $V \times_U V$). In other words, a morphism $U \to \@S(\@F)$ can be represented by data which consists of morphisms and homotopies into $\@F$.  We now state a criterion for two elements of $\@S(\@F)(U)$ to map to the same element in $\pi_0^{\A^1}(U)$, proved in \cite[Lemma 4.1]{Balwe-Hogadi-Sawant}, which will play a crucial role in our proof of Theorem \ref{main theorem}. 

\begin{lemma}
\label{lemma S^2 agrees with pi0}
Let $\sF$ be a Nisnevich sheaf of sets over $Sm/k$.  Let $U$ be a smooth scheme over $k$ and let $f,g: U \to \@F$ be two morphisms. Suppose that we are given data of the form
\[
 (\{p_i: V_i \to \A^1_U\}_{i \in \{1,2\}}, \{\sigma_0,\sigma_1\}, \{h_1, h_2\}, h^W)
\]
where:
\begin{itemize}
 \item The two morphisms $\{p_i: V_i \to \A^1_U\}_{i=1,2}$ constitute an elementary Nisnevich cover.
 \item For $i \in \{0,1\}$, $\sigma_i$ is a morphism $U \to V_1 \coprod V_2$ such that $(p_1 \coprod p_2) \circ \sigma_i: U \to U \times \A^1$ is the closed embedding $U \times \{i\} \hookrightarrow U \times \A^1$,
 \item For $i \in \{1,2\}$, $h_i$ denotes a morphism $V_i \to \@F$ such that $(h_1 \coprod h_2) \circ \sigma_0 = f$ and $(h_1 \coprod h_2) \circ \sigma_1 = g$. 
 \item Let $W := V_1 \times_{\A^1_U}V_2$ and let $pr_i: W \to V_i$ denote the projection morphisms. Then $h^W = (h_1, \ldots, h_n)$ is a $\A^1$-chain homotopy connecting the two morphisms $h_i \circ pr_i: W  \to \@F$. 
\end{itemize}
Then, $f$ and $g$ map to the same element under the map $\@F(U) \to \pi_0^{\A^1}(\@F)(U)$.  
\end{lemma}

We now set up the notation for general data of such type in order to represent homotopies in $\@S^n(\@F)$.  We will only need to do this for the cases $n = 1$ and $2$ and when $\@F$ is of the form $\@S^n(X)$ for a scheme $X$ but the inductive definition is easy enough to state in general.

\begin{definition}
\label{definition ghost homotopy}
Let $\@F$ be a sheaf of sets and let $U$ be an essentially smooth scheme over $k$. Let $n \geq 0$ be an integer. The notion of an \emph{$n$-ghost homotopy} and \emph{$n$-ghost chain homotopy} is defined as follows:
\begin{itemize}
 \item[(1)] A \emph{$0$-ghost homotopy} is the same as an $\#A^1$-homotopy as defined in Definition \ref{definition homotopy}. Similarly, a $0$-ghost chain homotopy is the same as an $\#A^1$-chain homotopy. 
 \item[(2)] Assuming that the notion of $m$-ghost homotopy and $m$-ghost chain homotopy has been defined for $m < n$, we define an $n$-ghost homotopy.  Given $t_1, t_2 \in \@F(U)$, an \emph{$n$-ghost homotopy} connecting $t_1, t_2$ consists of the data:
\[
\@H:=(V \to \#A^1_U, W \to V \times_{\#A^1_U} V, \~\sigma_0, \~\sigma_1, h, h^W )
\]
which is defined as follows:
\begin{itemize}
\item[(a)] $V \to \#A^1_U$ is a Nisnevich cover of $\#A^1_U$.
\item[(b)] For $i = 0,1$, $\~\sigma_i: U \to V$ is a lift of $\sigma_i: U \to \#A^1_U$. 
\item[(c)] $W \to V \times_{\#A^1_U} V$ is a Nisnevich cover of $V \times_{\#A^1_U} V$.  
\item[(d)] $h$ is a morphism $V \to \@F$ such that $h \circ \~{\sigma_i} = t_{i+1}$. 
\item[(e)] $h^W$ is an $(n-1)$-ghost chain homotopy connecting the two morphisms $W \to V\times_{\#A^1_U} V  \overset{pr_i}{\to} V \to \@F$ where $pr_1$ and $pr_2$ are the projections $V \times_{\#A^1_U} V \to V$. 
\end{itemize}
With this notation, we will also write $\@H(i) = t_{i+1}$ for $i = 0,1$. 
\item[(3)] Suppose the notion of an $n$-ghost homotopy has been defined. Then for elements $t_1, t_2 \in \@F(U)$, an \emph{$n$-ghost chain homotopy} connecting $t_1, t_2$ is a finite sequence $\@H=(\@H_1, \ldots, \@H_r)$ where each $\@H_i$ is an $n$-ghost homotopy of $U$ in $\@F$ such that $\@H_i(1) = \@H_{i+1}(0)$ for $1 \leq i \leq r-1$, $\@H_1(0) = t_1$ and $\@H_r(1) = t_2$.
\end{itemize}
\end{definition}

\begin{lemma}
Let $\@F$ be a sheaf of sets, let $U$ be an essentially smooth scheme over $k$ and let $n \geq 0$ be an integer. Let $t_1, t_2 \in \@F(U)$. If $t_1, t_2$ are $n$-ghost homotopic then their images in $\@S^n(\@F)(U)$ are $\#A^1$-homotopic. Conversely, if the images of $t_1, t_2$ are $\#A^1$-homotopic, there exists a Nisnevich cover $p: U^{\prime} \to U$ such that the images of $t_i \circ p$ in $\@S^{n}(\@F)(U)$ are $n$-ghost homotopic.
\end{lemma}
\begin{proof}
We first prove that if $t_1$ and $t_2$ are $n$-ghost homotopic, then their images in $\@S^n(\@F)$ are $\#A^1$-homotopic. We prove this by induction on $n$. The case $n = 0$ is trivial. We verify the case $n = 1$.  Suppose we have a $1$-ghost homotopy 
\[
\@H:=(V \to \#A^1_U, W \to V \times_{\#A^1_U} V, \~\sigma_0, \~\sigma_1, h, h^W )
\]
connecting $t_1$ and $t_2$. Then for $i =1,2$, if $pr_i$ denotes the projection on the $i$-th factor, we see that the two morphisms 
\[
W \to  V \times_{\#A^1_U} V \stackrel{pr_i}{\to} V \stackrel{h}{\to} \@F \to \@S(\@F)
\]
are equal, by the definition of $\@S(\@F)$. Since $W \to V \times_{\#A^1_U} V$ is a Nisnevich cover, we see that the morphism $V \to \@F \to \@S(\@F)$ descends to a morphism $\#A^1_U \to \@F$, connecting the images of $t_1$ and $t_2$ in $\@S(\@F)$. 

Now assume that for any scheme $T$, if two morphisms $T \to \@F$ are $(n-1)$-ghost homotopic, then their images in $\@S^{n-1}(\@F)$ are $\#A^1$-homotopic. Suppose we are given an $n$-ghost homotopy
\[
\@H:=(V \to \#A^1_U, W \to V \times_{\#A^1_U} V, \~\sigma_0, \~\sigma_1, h, h^W )
\]
between $t_1,t_2 \in \@F(U)$.  By the induction hypothesis, on composing with the morphism $\@F \to \@S^{n-1}(\@F)$ we get a $1$-ghost homotopy of $U$ in $\@S^{n-1}(\@F)$. Since we know the result to be true for $n = 1$, we see that a $1$-ghost homotopy in $\@S^{n-1}(\@F)$ gives rise to an $\#A^1$-homotopy of $U$ in $\@S^n(\@F)$ connecting the images of $t_1$ and $t_2$ in $\@S^n(\@F)(U)$. 

Now we prove the converse, again by induction on $n$.  The case $n = 0$ is trivial. Suppose $t_1, t_2 \in \@F(U)$ are such that their images in $\@S^n(\@F)(U)$ are connected by a single $\#A^1$-homotopy. As $\@F \to \@S^n(\@F)$ is an epimorphism, there exists a Nisnevich cover $V \to \#A^1_U$ such that the morphism $V \to \#A^1_U \to \@S^n(\@F)$ can be lifted to a morphism $V \to \@F$. By replacing $U$ by some suitable Nisnevich cover $U' \to U$, we may assume that the morphisms $\sigma_i: U \to \#A^1_U$ lift to $V$. Now, in the notation of Definition \ref{definition ghost homotopy}, it suffices to construct the Nisnevich cover $W \to V \times_{\#A^1_U} V$ and the $(n-1)$-ghost chain homotopy $\@H^W$.

The two morphisms 
\[
V \times_{\#A^1_U} V \stackrel{pr_i}{\to} V \stackrel{h}{\to} \@F \to \@S^{n-1}(\@F)
\]
become equal when composed with the morphism $\@S^{n-1}(\@F) \to \@S^n(\@F)$. Thus, they must be compatible up to $\#A^1$-homotopy in $\@S^{n-1}(\@F)$, that is, there exists a Nisnevich cover $W \to V \times_{\#A^1_U} V$ such that the two morphisms 
\[
W \to  V \times_{\#A^1_U} V \stackrel{pr_i}{\to} V \stackrel{h}{\to} \@F \to \@S^{n-1}(\@F)
\]
are $\#A^1$-chain homotopic in $\@S^{n-1}(\@F)$.  Hence, there exists a finite sequence of homotopies $h^W:= (h_1, \ldots, h_r)$ in $\@S^{n-1}(\@F)$ such that $\sigma_1^*(h_i) = \sigma_0^*(h_{i+1})$ for $1 \leq i \leq r-1$ and such that $\sigma_0^*(h_1)$ and $\sigma_1^*(h_r)$ are morphisms $W \to \@S^{n-1}(\@F)$.  By replacing $W$ by some suitable Nisnevich cover, we may assume that for every $i$, $1 \leq i \leq r-1$, the morphisms $\sigma_1^*(h_i)$ lift to $\@F$.  Applying the induction hypothesis, we see that all these lifts are $(n-1)$-ghost homotopic. Thus, the two morphisms $W \to \@F$ are $(n-1)$-ghost chain homotopic. This completes our proof.
\end{proof}

In order to avoid being flooded by notation in some of the proofs, we need to introduce the notion of the ``total space'' of an $n$-ghost homotopy. For a given $n$-ghost homotopy $\@H$ of a scheme $U$ in a sheaf $\@F$, this is simply the union of all the schemes that show up in its definition. This is a scheme over $U$ and is equipped with a morphism into $\@F$ which is simply the coproduct of all the morphisms that come up in the definition of $\@H$. For the sake of precision, we write down the definition explicitly as follows:

\begin{definition}
\label{definition ghost homotopy total space}
Let $\@F$ be a sheaf of sets, let $U$ be a smooth scheme. For any $n$-ghost homotopy $\@H$ of $U$ in $\@F$, we define a scheme $\Sp(\@H)$ and morphisms $f_{\@H}: \Sp(\@H) \to U$ and $h_{\@H}: \Sp(\@H) \to \@F$. We do this by induction on $n$ as follows: 
\begin{itemize}
\item[(1)] For a $0$-ghost homotopy $\@H$ which is given by a morphism $\#A^1_U \to \@F$, we define $$\Sp(\@H) := \#A^1_U = \#A^1 \times U.$$ The morphism $f_{\@H}: \#A^1_U \to U$ is simply the canonical projection on $U$. The morphism $h_{\@H}: \#A^1_U \to \@F$ is just the morphism defining the homotopy.
\item[(2)] Suppose that this construction has been done for $(n-1)$-ghost homotopies. Let 
\[
\@H:=(V \to \#A^1_U, W \to V \times_{\#A^1_U} V, \~\sigma_0, \~\sigma_1, h, \@H^W )
\]
be an $n$-ghost homotopy where $\@H_W = (\@H_1, \ldots, \@H_r)$ is a $(n-1)$-ghost chain homotopy.  We then define 
\[
\Sp(\@H) := V \amalg \left(\coprod_{i=1}^r\Sp(\@H_i)\right).
\]
We will define the morphisms $f_{\@H}:\Sp(\@H) \to U$ and $h_{\@H}:\Sp(\@H) \to \@F$ are defined by specifying their restrictions to $V$ and to $\Sp(\@H_i)$ for all $i$.  We define $f_{\@H}|_V$ to be the composition of the given morphism $V \to \#A^1_U$ with the projection $\#A^1_U \to U$. The morphism $f_{\@H}|_{\Sp(\@H_i)}$ is defined to be the composition
\[
\Sp(\@H_i) \stackrel{f_{\@H_i}}{\to} W \to V \times_{\#A^1_U} V \stackrel{pr_i}{\to} V \to \#A^1_U \to U
\]
where $pr_i$ could be either of the projections $V \times_{\#A^1_U} V \to V$ (the composition clearly does not depend on this choice). The morphism $h_{\@H}:\Sp(\@H) \to \@F$ is defined by $h_{\@H}|_{V} = h$ and $h_{\@H}|_{\Sp(\@H_i)} = h_{\@H_i}$. 
\end{itemize}
\end{definition}

\subsection{Ghost homotopies respecting fibers}
\label{subsection ghost homotopies respecting fibers}

We now discuss $\A^1$-homotopies and $n$-ghost homotopies on sheaves that admit a morphism to an $\A^1$-invariant sheaf.  We will then specialize to the case when the target sheaf is represented by an $\A^1$-rigid scheme.

\begin{definition}
\label{definition homotopy respecting fibres}
Let $U$ be a scheme over $k$. Let $\phi: \mathcal F \to \mathcal G$ be a morphism of sheaves of sets on $Sm/k$.  
\begin{enumerate}[label=$(\arabic*)$]
\item We say that an $\mathbb A^1$-homotopy $h: U \times \mathbb A^1 \to \mathcal F$ \emph{respects fibers of $f$} if there exists a morphism $\gamma: U \to \mathcal G$ such that $\phi \circ h = \gamma \circ pr_1$ where $pr_1: U \times \mathbb A^1 \to U$ is the projection on the first factor. We say that \emph{$h$ lies over $\gamma$}.  
\item Let $\mathcal H:=(\mathcal V, \{\mathcal W^{ij}\}_{i,j}, \{h_i\}_i, \{h^{ijk}\}_{i,j,k})$ be an $\mathbb A^1$-ghost homotopy of $U$ in $\mathcal F$. We say that $\mathcal H$  \emph{respects fibres of $f$} if there exists a morphism $\gamma: U \to \mathcal G$ such that $\phi \circ h_{\mathcal H} = \gamma \circ f_{\mathcal H}$. We say that \emph{$\mathcal H$ lies over $\gamma$}.
\item In general, given a $U$-scheme $V$ and a morphism of sheaves $\@F \to \@G$, a morphism $V \to \@F$ will be called a \emph{$\gamma$-morphism} if the diagram
\[
\xymatrix{
V \ar[r] \ar[d] & \@F \ar[d] \\
U \ar[r]^-{\gamma}        & \@G
}
\]
commutes.
\end{enumerate}
\end{definition}

\begin{remark}
An $\mathbb A^1$-ghost homotopy of $U$ in $\mathcal F$ gives rise to an $\mathbb A^1$-homotopy of $U$ in $\mathcal S(\mathcal F)$. However, for an $\mathbb A^1$-rigid sheaf $\mathcal G$, to say that an $\mathbb A^1$-ghost homotopy respects fibers of $\mathcal F \to \mathcal G$ is, in general, stronger than saying that the corresponding $\mathbb A^1$-homotopy in $\mathcal S(\mathcal F)$ respects the fibers of $\mathcal S(\mathcal F) \to \mathcal G$. For instance, if $\mathcal H$ is an $\mathbb A^1$-ghost homotopy of $U$ in $\mathcal F$ such that the corresponding $\mathbb A^1$-homotopy of $U$ in $\mathcal S(\mathcal F)$ respects the fibers of the morphism $\mathcal S(\mathcal F) \to \mathcal G$, it need not be true that $\phi \circ h_{\mathcal H}  = \gamma \circ f_{\mathcal H}$ on every copy of $W^{ij}_k \times \mathbb A^1$ that occurs in $\Sp(\mathcal H)$. 
\end{remark}

\begin{lemma}
\label{lemma strategy ruled}
Let $\phi: \mathcal F \to \mathcal G$ be a morphism of sheaves of sets on $Sm/k$.  Assume that $\mathcal G$ is $\mathbb A^1$-invariant. Let $U$ be any scheme.  Then any $\mathbb A^1$-homotopy (resp. $n$-ghost homotopy) of $U$ in $\mathcal F$ respects fibers of $\phi$. Thus there exists a morphism $\gamma: U \to \mathcal G$ such that the given $\mathbb A^1$-homotopy (resp. $n$-ghost homotopy) of $U$ factors through $\mathcal F \times_{\mathcal G,\gamma} U \to \mathcal F$. 
\end{lemma}

\begin{proof}
Let $h: U \times \mathbb A^1 \to \mathcal F$ be an $\mathbb A^1$-homotopy of $U$ in $\mathcal F$. Since $\mathcal G$ is $\mathbb A^1$-invariant, the $\mathbb A^1$-homotopy $f \circ h$ is constant, that is, there exists a morphism $\gamma: U \to \mathcal G$ such that $\phi \circ h = \gamma \circ pr_1$. This proves the lemma for $\mathbb A^1$-homotopies. 

Now, suppose we have an $\mathbb A^1$-ghost homotopy 
\[
\mathcal H:=(\mathcal V, \{\mathcal W^{ij}\}_{i,j}, \{h_i\}_i, \{h^{ijk}\}_{i,j,k})
\]
of $U$ in $\mathcal F$ which represents an $\mathbb A^1$-homotopy $h$ of $U$ in $\mathcal S(\mathcal F)$. Since we have proved the result for $\mathbb A^1$-homotopies, the fact that $\mathcal G$ is $\mathbb A^1$-invariant implies that there exists a $\gamma \in \mathcal G(U)$ giving a commutative diagram
\[
\xymatrix{
U \times \mathbb A^1 \ar[r]^{h} \ar[d]_{pr_1} & \mathcal S(\mathcal F) \ar[d] \\
U \ar[r]^{\gamma} & \mathcal G  \text{.}
}
\]
Combining this with the commutative diagram 
\[
\xymatrix{
V_i \ar[r]^{h_i} \ar[d] & \mathcal F \ar[d]  \\
U \times \mathbb A^1  \ar[r]^h & \mathcal S(\mathcal F)
}
\]
we obtain a proof of the equality $\phi \circ h_{\mathcal H} = \gamma \circ f_{\mathcal H}$ when the morphisms are restricted to $V_i \subset\Sp(\mathcal H)$ for $i \in I$. 

Let $i,j \in I$ and $k \in K_{ij}$. Since $\mathcal G$ is $\mathbb A^1$-invariant, for every $\mathbb A^1$-homotopy $h^{ijk}_l: W^{ij}_k \times \mathbb A^1 \to \mathcal F$ occurring in $h^{ijk}$, there exists a morphism $\gamma_l: W^{ij}_k \to \mathcal G$ such that the square 
\[
\xymatrix{
W^{ij}_k \times \mathbb A^1 \ar[r]^-{pr_1} \ar[d]_-{h_{l}^{ijk}} & W^{ij}_k \ar[d]^-{\gamma_l} \\
\mathcal F \ar[r]^-{\phi} & \mathcal G 
}
\]
is commutative. We need to show that $\gamma_l$ is the same as the morphism 
\[
W^{ij}_k \to U \times \mathbb A^1 \overset{pr_1}{\to} U \overset{\gamma}{\to} \mathcal G.
\]
For $t=0,1$, we have $\phi \circ h^{ijk}_l\circ\sigma_t =  \gamma_l \circ pr_1 \circ \sigma_t = \gamma_l$. However, by the definition of an $\mathbb A^1$-chain homotopy, we have $h^{ijk}_l \circ \sigma_0 = h^{ijk}_{l-1} \circ \sigma_1$ and $h^{ijk}_l \circ \sigma_1 = h^{ijk}_{l+1} \circ \sigma_0$. Thus we see that the $\gamma_l = \gamma_{1}$ for every $l$. However, $\gamma_1$ is equal to the composition of morphism
\[
W^{ij}_k \overset{\sigma_0}{\to} W^{ij}_k \times \mathbb A^1 \overset{h^{ijk}_l}{\to} \mathcal F \to \mathcal G
\]
which, in turn, is equal to the composition 
\[
W^{ij}_k \to V_i \times_{\mathbb A^1_U} V_j \overset{pr_1}{\to} V_i \overset{h_i}{\to} \mathcal F \overset{\phi}{\to} \mathcal G \text{.}
\]
Since the morphism $\phi \circ h_i$ is the same as the composition 
\[
V_i \to U \times \mathbb A^1 \overset{pr_1}{\to} U \overset{\gamma}{\to} \mathcal G \text{,}
\]
we see that $\gamma_1$ is equal to the composition
\[
W^{ij}_k \to V_{i} \times_{\mathbb A^1_U} V_j \overset{pr_1}{\to} V_i \to U \times \mathbb A^1 \overset{pr_1}{\to} U \overset{\gamma}{\to} \mathcal G
\]
as desired.   Proceeding inductively, we can prove the result for $n$-ghost homotopies.
\end{proof}

\begin{proposition}
\label{proposition P1 bundle over an A1-rigid scheme}
Let $B$ be a smooth projective $\A^1$-rigid scheme over a field $k$.
\begin{enumerate}[label=$(\alph*)$]
\item If $E \to B$ is the projectivization of a rank $2$ vector bundle on $B$, then $\mathcal S(E) = \pi_0^{\mathbb A^1}(E) = B$.

\item If $E \to B$ is a $\mathbb P^1$-bundle, $B$ is of dimension $1$ and $k$ is algebraically closed of characteristic $0$, then $\mathcal S(E) = \pi_0^{\mathbb A^1}(E) = B$.
\end{enumerate}
\end{proposition}
\begin{proof}
We fix a smooth henselian local ring $(R, \mathfrak{m})$, $U = \Spec R$. We fix a morphism $\gamma: U \to B$ and consider the pullback $E$ with respect to $\gamma$ which we denote by $E_{\gamma}$.  For part $(a)$, note that since $E_{\gamma} \simeq \mathbb P^1_U$ and $B$ is $\mathbb A^1$-rigid, by Lemma \ref{lemma strategy ruled}, it follows that $\mathcal S(E)(U) = B(U)$, and consequently,
\[
\mathcal S(E) = \pi_0^{\mathbb A^1}(E) \simeq B.
\] 
For part $(b)$, observe that a $\mathbb P^1$-bundle is \'etale locally trivial (see \cite[Theorem 3.4]{Beauville}).  Hence, $E_{\gamma} \simeq \mathbb P^1_U$ and we can conclude as in part $(a)$.
\end{proof}

The above proposition proves part (a) of Theorem \ref{theorem ruled iterations of S}. In order to prove part (b), we need to show that if $X$ is a smooth projective surface, birationally ruled over a curve $C$ of genus $>0$, then for any henselian local scheme $U$, essentially smooth over $k$, we have $\pi_0^{\#A^1}(X)(U) = \@S^n(X)(U)$ for all $n \geq 2$. Thus, we need to show that if $\alpha_1, \alpha_2: U \to X$ are two morphisms which are $n$-ghost homotopic, then they are actually $1$-ghost homotopic and also have the same image in $\pi_0^{\#A^1}(X)(U)$. The following lemma disposes of two cases in which this is trivially true:

\begin{lemma}
\label{lemma trivial case}
Let $X$ be a smooth projective surface, birationally ruled over a curve $C$ of genus $>0$. Let $U$ be a henselian local scheme, essentially smooth over $k$. Let $\gamma: U \to C$ be a morphism and let $\alpha_1, \alpha_2: U \to X$ be $\gamma$-morphisms which are $n$-ghost homotopic. Then $\alpha_1$ and $\alpha_2$ are $\#A^1$-chain homotopic if either of the following conditions hold:
\begin{itemize}
\item[(a)] $\gamma$ maps the generic point of $U$ to a closed point of $C$. 
\item[(b)] $\gamma$ maps the closed point of $U$ to the generic point of $C$. 
\end{itemize}
Hence, if either of these two conditions holds, $\alpha_1$ and $\alpha_2$ map to the same elements of $\pi_0^{\#A^1}(X)(U)$ and $\@S^n(X)(U)$ for any $n \geq 1$. 
\end{lemma}

\begin{proof}
In both cases, there exists a point $c$ of $C$ such that $\alpha_1$ and $\alpha_2$ factor through the fiber $X_c:= X \times_{C,c} \Spec \kappa(c)$. Indeed, in case (a), $c$ is a closed point of $C$ and in case (b), it is the generic point of $C$. The result follows immediately from the fact that in both cases the fiber $X$ is connected and its components are isomorphic to $\#P^1_{\kappa(c)}$.  
\end{proof}

\begin{remark}
\label{remark trivial case}
Due to Lemma \ref{lemma trivial case}, we will primarily be concerned with the case in which $\gamma$ maps the generic point of $U$ to the generic point of $C$ and the closed point of $U$ to the closed point of $C$.  
\end{remark}

\section{Homotopies on a blowup}
\label{section homotopies on a blowup}

Let $X$ be a smooth projective surface, birationally ruled over a curve $C$ of genus $>0$. Let $E$ be a minimal surface such that there exists birational morphism $X \to E$. Thus $E$ is a $\#P^1$-bundle over $C$. Let $\pi$ denote the composition $X \to E \to C$, where the second map is the projection map of the $\#P^1$-bundle. Let $\~X$ be obtained from $X$ by blowing up a point $p$ and suppose that $\pi(p) = c$. 

Let $U$ be an essentially smooth local scheme of dimension $1$, with closed point $u$. Let $h: U \times \#A^1 \to X$ be a morphism which lifts to $\~X$. For $i = 0,1$ let $\alpha_i: U \to X$ be the morphism $h(-,i): U \to X$. As $C$ is rigid, we know from the above discussion, there exists a morphism $\gamma: U \to C$ such that $h$ is a $\gamma$-morphism. We assume that $\gamma$ maps $u$ to $c$ and the generic point of $U$ to the generic point of $C$.  Since $h$ lifts to $\~X$, the scheme-theoretic preimage $h^{-1}(p)$ is a locally principal divisor on $U \times \#A^1$. Also, it is supported on $\{u\} \times \#A^1$ which is an irreducible codimension $1$ subscheme of $U \times \#A^1$. Thus, it follows that the support of $h^{-1}(p)$ is either empty or it is equal to $\{u\} \times \#A^1$. Thus, if $\alpha_0$ maps $u$ to $p$, so does $\alpha_1$. In fact, in this case, the homotopy $h$ maps $u$ constantly to $p$. On the other hand, if $\alpha_0$ does not map $u$ to $p$, then both $\alpha_1$ and $h$ do not have $p$ in their images. 

In this section, we will strengthen the above observation and explore its consequences. We will prove that the homotopy $h$ in the above paragraph can be replaced by an $n$-ghost homotopy for any $n \geq 0$. The discussion in the previous paragraph easily shows that the scheme theoretic preimages $\alpha_0^{-1}(p)$ and $\alpha_1^{-1}(p)$ have the same support on $U$. We will obtain the stronger statement that these two preimages are actually the same. The proof of this statement is somewhat technical and will occupy Section \ref{subsection general result}. In Section \ref{subsection homotopies surfaces}, we will explore the consequences of the result in Section \ref{subsection general result} in the context of birationally ruled surfaces.
   
\subsection{A general result}

\label{subsection general result}

We fix the following setting for the rest of this subsection.

\begin{notation}
\begin{enumerate}[label=$(\arabic*)$]
\item Let $k$ be a field and let $B$ be an $\#A^1$-rigid scheme over $k$ and let $\pi:X \to B$ be a morphism of schemes over $k$. 

\item Let $\~X$ be the blowup of $X$ at $T$, where $T$ is a closed subscheme of $X$ such that $\pi(T)$ is not dense in $B$.

\item Let $U = \Spec R$ where $(R, \fr m)$ is the henselization of the local ring at a smooth point of a variety over $k$.  Let $K$ denote the residue field $R/\fr{m}$.  We will denote the closed point of $U$ by $u$.

\item For any morphism of schemes $f: Y \to Z$, we will abuse the notation and write $f^*$ for the associated morphism of sheaves of rings $\@O_Z \to f_*\@O_Y$ as well as for the morphism induced by this one on sections, when there is no confusion.
\end{enumerate}
\end{notation}

Let $\alpha_1, \alpha_2: U \to X$ be morphisms which lift to $\~X$. We also assume that the maps $\pi \circ \alpha_i$ do not map the generic point of $U$ into $\pi(T)$.  This condition is enough to ensure that the lifts of $\alpha_1, \alpha_2$ to $\~X$ are unique. The main result of this section is that, if the lifts $\~{\alpha}_1$ and $\~{\alpha}_2$ of $\alpha_1$ and $\alpha_2$ to $\~X$ are $n$-ghost homotopic for some $n$, then the intersection of the exceptional divisor with $\~{\alpha}_1(U)$ and $\~{\alpha}_2(U)$ is the ``same" in the sense that the pullbacks of the ideal sheaves defining $T$ to $U$ via $\alpha_1$ and $\alpha_2$ are the same.  For instance, this implies that if $\~{\alpha}_1$ maps the closed point of $U$ into the exceptional divisor, so does $\~{\alpha}_2$. Indeed, the image of the closed point of $U$ remains in the exceptional divisor throughout the $n$-ghost homotopy. 

We wish to study the sheaves $\@S^n(X)$. For any $n$-ghost homotopy $\@H$ of $U$ in $X$, the homotopy $\pi \circ \@H$ is constant by Lemma \ref{lemma strategy ruled}. Thus, if two morphisms $\alpha_1, \alpha_2: U \to X$ are $n$-ghost chain homotopic, the compositions $\pi \circ \alpha_1$ and $\pi \circ \alpha_2$ are equal. Thus, in order to study the $n$-ghost homotopy classes of morphisms $U \to X$, we may first fix a morphism $\gamma: U \to B$ and study the $n$-ghost homotopy classes of the sections of $X \times_{B,\gamma} U \to U$.  Therefore, we may take $B = U$ without loss of generality throughout the rest of the section. 

Let us thus assume that $X$ is a scheme of finite type over $U$, with structure map $\pi: X \to U$.  Let $T$ be a closed subscheme of $X$ such $\pi(T)$ does not contain the generic point of $U$.  Thus, there exists a proper ideal $I_0$ of $R$ such that if $\@I_0$ is the associated ideal sheaf on $U$, then $T$ is contained in the closed subscheme defined by the ideal sheaf $f^*(\@I_0)$. Let $\@I_T$ denote the ideal sheaf corresponding to $T$ and $\~X$ denote the blowup of $X$ at $T$.

\begin{lemma}
\label{lemma local generator}
Let $f: V \to U$ be a smooth morphism. Let $h: V \to X$ be a morphism over $U$ that lifts to $\~X$. Let $v$ be a point of $V$ such that $f(v) = u$ and $\kappa(v) = K$. Then the ideal $h^*(\@I_T)_v$ of $\mathcal O_{V,v}$ is generated by an element of the form $f^*(r)$ for some $r \in R$.  Moreover, we have $\<r\> \supset I_0$. 
\end{lemma}
\begin{proof}
Since $R$ is henselian, there exists a section $\alpha: U \to V$ such that $\alpha(u) = v$. Since $h \circ \alpha: U \to X$ lifts to $\~X$, the ideal $\alpha^*(h^*(\@I_T)_v)$ is principal. If $h(v)$ does not lie in $T$, then this ideal is the unit ideal and so there is nothing left to prove. So we now assume that $h(v) \in T$. Thus the ideal $\alpha^*(h^*(\@I_T)_v)$ is generated by some element $r \in \fr{m}$. It is clear that $\<r\> \supset I_0$. We will show that $h^*(\@I_T)_v$ is generated by $f^*(r)$. As $h$ lifts to $\~X$, we know that the ideal $h^*(\@I_T)_v$ is principal, generated by some element $\rho \in \@O_{V,v}$. Thus we need to show that $f^*(r)$ is a unit multiple of $\rho$.

The sequence of homomorphisms $ R \overset{f^*}{\to} \@O_{V,v} \overset{\alpha^*}{\to} R$ gives rise to the sequence of homomorphisms on the completions $\^{R} \overset{\^{f}^*}{\to} \^{\@O}_{V,v} \overset{\^{\alpha}^*}{\to} \^{R}$ and $\^{\alpha}^* \circ \^{f}^*$ is the identity homomorphism on $\^{R}$. (Here $\^f$ and $\^{\alpha}$ are the induced morphisms of schemes $\Spec \^{\@O}_{V,v} \to \Spec \^{R}$ and $\Spec \^R \to \Spec \^{\@O}_{V,v}$ respectively.) We denote by $\phi_R: \Spec \^R \to \Spec R$ and $\phi_v: \Spec \^{\@O}_{V,v} \to \Spec \@O_{V,v}$ the morphisms induced by the canonical homomorphisms of local rings into their completions.

Suppose $n$ is the Krull dimension of $R$ and $m$ is the dimension of the fibers of $f$. Then there exists a Zariski local neighbourhood $W$ of $v$ in $V$ such that $f|_W$ factors as $W \to \#A^n_U \to U$ where $W \to \#A^n_U$ is an \'etale morphism taking $v$ to the origin in $\#A^n_K \subset \#A^n_U$. Thus there exists a commutative diagram as follows where the horizontal morphisms are isomorphisms:
\[
\xymatrix{
\^{R} \ar[r] \ar[d]_{\^{f^*}} & K[[s]] \ar@{^{(}->}[d] \\
\^{\@O}_{V,v} \ar[r] &   K[[s,t]] \text{.}
}
\]
Here $s = (s_1, \ldots, s_n)$ and $t = (t_1, \ldots, t_m)$ are tuples of variables. 

Let $r_0$ be a non-zero element in $I_0$. (Such an element exists because we are assuming that the generic point of $U$ does not lie in $\pi(T)$.) The ideal $\^{h^*(\@I_T)}_v = \phi_v^*(h^*(\@I_T)_v)$ in $\^{\@O}_{V,v}$ is principal, generated by $\phi_v^*(\rho)$. Then $\phi_v^*(\rho)| \^{f}^*(r_0)$ in $\^{\@O}_{V,v}$ and thus by the above commutative square we see that $\phi_v^*(\rho)$ is a unit multiple of some element in the image of $\^{f}^*$. Thus $\phi_v^*(h^*(\@I_T)_v) = \<\^{f}^*(r^{\prime})\>$ where $r^{\prime} \in \^R$. 

The morphism $\alpha \circ \phi_R: \Spec \^R\to V$, is equal to the composition of $\phi_v \circ \^{\alpha}$ with the obvious morphism $\Spec \@O_{V,v} \to V$. Thus  we have
\begin{eqnarray*}
(\alpha \circ \phi_R)^*(h^*(\@I_T)_v) & = & \^{\alpha}^*(\phi_v^*(h^*(\@I_T)_v)) \\      
& = & \^{\alpha}^*(\<\^{f}^*(r^{\prime})\>)\\
& = & \<r^{\prime}\> \text{.}
\end{eqnarray*}
However, we also have $(\alpha \circ \phi_R)^*(h^*(\@I_T)_v) = \phi_R^*(\<r\>)$. Thus $r^{\prime}$ is a unit multiple of $\phi_R^*(r)$ in $\^R$. Thus the ideals $\phi_v(h^*(\@I_T)_v) = \<\^{f^*}(r^{\prime})\>$ and $\phi_v(\<f^*(r)\>) = \<\^{f^*}(\phi_R^*(r))\>$ are equal. Since $\phi_v$ is a faithfully flat morphism, this shows that $h^*(\@I_T)_v = \<f^*(r)\>$ in $R$.  
\end{proof}

In the above lemma, the element $r \in R$ depends on the point $v$. We will remove this dependence on $v$ through Lemmas \ref{lemma local generator along closed fibre}, \ref{lemma extending local generator from closed subscheme} and Proposition \ref{proposition local generator ghost homotopy}. First, we recall some elementary facts. While these may be well-known to experts, we present the details for the lack of a suitable reference as well as for the sake of completeness.  

\begin{lemma}
\label{lemma henselian ring coefficient field}
Let $S$ be a smooth scheme over a perfect field $F$ and let $s$ be a point of $S$. Then, the henselization $\@O_{S,s}^h$ of the local ring of $S$ at $s$ contains a field $L$ which maps isomorphically onto the residue field $\kappa(s)$ under the map $\@O_{S,s}^h \to \kappa(s)$. 
\end{lemma} 

\begin{proof}
We may assume that $S$ is of pure dimension $n$ over $F$. Let $s$ be a point of $S$ of codimension $r$ and let $Z$ be the Zariski closure of $s$ in $S$. Since $Z$ is generically smooth, there exists an open subscheme $W$ of $S$ and regular functions $f_1, \ldots, f_n: W \to \#A^n_F$ such that the map 
\[
f:= (f_1, \ldots, f_n): U \to \#A^n_F = \Spec F[X_1, \ldots, X_n] 
\]
is \'etale and $W \cap Z = f^{-1}(T)$ where $T$ is the linear subvariety of  $\#A^n_F$ defined by $X_1 = \cdots = X_r = 0$.  Hence, we obtain an \'etale ring homomorphism $\@O_{\#A^n_F, T} \to \@O_{S,s}$. Observe that the local ring $\@O_{\#A^n_F, T}$ contains the field $L_0:= F(X_{r+1}, \ldots, X_n)$, which maps isomorphically onto its residue field. Thus, we have an inclusion $L_0 \hookrightarrow \@O_{\#A^n_F,T} \hookrightarrow \@O_{S,s}$ and the composition $L_0 \hookrightarrow \@O_{S,s} \to \kappa(s)$ is a separable algebraic extension of finite degree. 

The ring homomorphism $\@O_{S,s} \to \@O_{S,s} \otimes_{L_0} \kappa(s)$, $\alpha \mapsto \alpha \otimes 1$ is \'etale and the morphism $s:\Spec \kappa(s) \to \Spec \@O_{S,s}$ admits a lift to $\~s: \Spec \kappa(s) \to \Spec(\@O_{S,s} \otimes_{L_0} \kappa(s)) = \~S$. Thus, we see that the local ring $\@O_{\~S, \~s}$ is an \'etale extension of $\@O_{S,s}$, satisfying $\kappa(s) = \kappa(\~s)$. Since this ring contains a field that maps isomorphically onto $\kappa(s)$, we conclude that the same is true of the henselization of $\@O_{S,s}$. 
\end{proof}

\begin{lemma}
\label{lemma openness}
Let $A \to B$ be a homomorphism of commutative rings. Then the morphism $\Spec B \to \Spec A$ is universally open in each of the following two cases:
\begin{itemize}
\item[(1)] $A$ is a field.
\item[(2)] $A$ is a noetherian local ring and the homomorphism $A \to B$ is the henselization of $A$ at its maximal ideal. 
\end{itemize} 
\end{lemma}

\begin{proof}
Case (1) is proved in \cite[Theorem 14.36]{GT}. The proof of Case (2) follows from essentially the same argument since the henselization $B$ is a direct limit of \'etale $A$-algebras of finite type. 
\end{proof}

\begin{lemma}
\label{lemma local generator along closed fibre} 
Let $f: V \to U$ be a smooth morphism. Let $h: V \to X$ be a morphism which lifts to $\~X$. Let $z$ be a point of $U$ and let $v$ be a point of $f^{-1}(z) = V \times_{U,z} \Spec \kappa(z)$ such that the ideal $h^*(\@I_T)_v$ in $\@O_{V,v}$ is generated by an element of the form $f^*(r)$ where $r \in R$. Let $Z$ be the irreducible component of $f^{-1}(z)$ containing $v$. Then there exists an open subscheme $W_0$ of $V$ such that $Z \subset W_0$ and the ideal sheaf $h^*(\@I_T)|_{W_0}$ is generated by $f^*(r)$. 
\end{lemma}

\begin{proof}
In the local ring $\@O_{V,v}$, we have the equality of ideals $h^*(\@I_T) = \<f^*(r)\>$ which must hold in a neighbourhood of $v$. Thus, there exists an open subscheme $W$ of $V$ with $v \in W$ such that the ideal sheaves $h^*(\@I_T)|_W$ and $\<f^*(r)\>$ of $\@O_V|_W$ are equal. Let $v_1 \in Z \backslash W$. We wish to show that the ideal $h^*(\@I_T)_{v_1}$ in $\@O_{V,v_1}$ is also generated by $f^*(r)$. 

We use Lemma \ref{lemma henselian ring coefficient field} to view $R_z^h$ as a $\kappa(z)$-algebra.  Using the ring homomorphism $f^*: \kappa(z) \to \kappa(v_1)$, we define the ring $R' := (\kappa(v_1) \otimes_{\kappa(z)} R_z^h)^h$.  Let $\psi_{U}$ denote the morphism $\Spec R' \to \Spec R$. Let $\~V = V \times_{U} \Spec R^{\prime}$. Let $\psi_V$ and $\~f$ denote the projections $\~V \to V$ and $\~V \to \Spec R^{\prime}$ respectively. 

We have a commutative square
\[
\xymatrix{
\Spec \kappa(v_1) \ar[r]^{v_1} \ar[d] & V \ar[d] \\
\Spec R' \ar[r] & \Spec R
}
\]
which gives a morphism $\~v_1: \Spec\kappa(v_1) \to \~V$. Then by the argument in the proof of Lemma \ref{lemma local generator}, we see that the ideal $(h \circ \psi_V)^*(\@I_T)_{\~v_1}$ in $\@O_{\~V, \~v_1}$ is generated by an element of the form $\~f^*(\rho)$ for some $\rho \in R^{\prime}$. Thus there exists an open subscheme $\~W$ of $\~V$ containing $\~v_1$ such that $(h \circ \psi_V)^*(\@I_T)(\~W)$ is generated by $\~f^*(\rho)$. 

By Lemma \ref{lemma openness}, $\psi_V$ is open. Thus, $\psi_V(\~W)$ is an open subscheme containing $v_1$. Thus, there exists a point $\~v_2$ in $\~W$ such that the point $v_2 := \psi_V(\~v_2)$ is in $W$. Thus 
\[(h \circ \psi_V)^*(\@I_T)_{\~v_2} = \psi_V^*(h^*(\@I_T)_{v_2}) = \psi_V^*(\<f^*(r)\>) = \<\~f^*(\psi_U^*(r))\> \text{.}
\] 
However, we also have $(h \circ \psi_V)^*(\@I_T)_{\~v_2} = \<\~f^*(\rho)\>$ as $\~v_2 \in \~W$. As $\~f$ is faithfully flat, we see that the ideals $\<\psi_U^*(r)\>$ and $\<\rho\>$ of $R'$ are equal. Thus $\rho$ is a unit multiple of $\psi_U^*(r)$ in $R'$. Thus the ideal $\psi_V^*(h^*(\@I_T)_{v_1})$ in $\@O_{\~V, \~v_1}$ is equal to $\psi_V^*(\<f^*(r)\>)$. Since $\psi_V^*: \@O_{V,v_1} \to \@O_{\~V,\~v_1}$ is faithfully flat, this implies that the ideal $h^*(\@I_T)_{v_1}$ in $\@O_{V,v_1}$ is equal to $\<f^*(r)\>$.  
\end{proof}

\begin{lemma}
\label{lemma extending local generator from closed subscheme}
For $i = 1,2$, let $f_i: V_i \to U$ be smooth morphisms and let $h_i: V_i \to X$ be morphisms over $U$ which lift to $\~X$. Let $f: V_1 \to V_2$ and $g: V_2 \to V_1$ be morphisms over $U$ such that $g$ is a closed embedding, $h_2 = h_1 \circ g$ and $f \circ g = Id_{V_2}$. Let $v_i$ be a point of $V_i$ for $i = 1,2$ such that $g(v_1) = v_2$. Suppose $h_2^*(\@I_T)_{v_2} = \<f_2^*(r)\>$ for some $r \in R$. Then $h_1^*(\@I_T)_{v_1} = \<f_1^*(r)\>$. 
\end{lemma}
\begin{proof}
This proof is along the same lines as the proof of Lemma \ref{lemma local generator along closed fibre}. Let $z = f_1(v_1) = f_2(v_2)$. Let $R_z$ be the localization of $R$ at the prime ideal corresponding to the point $z$ and let $R'$ be as in the proof of Lemma \ref{lemma local generator along closed fibre}. Let $\psi_U$ be the morphism $\Spec R' \to \Spec R$.  For $i=1,2$, we define $\~V_i = V_i \times_{U} \Spec R'$ and denote by $\~f_i$ and $\psi_{V_i}$ the projections $\~V_i \to \Spec R'$ and $\~V_i \to V_i$, respectively. Also, the morphisms $\~f$ and $\~g$ are the pullbacks of $f$ and $g$. We also have points $\~v_i$ on $\~V_i$ for $i = 1,2$ such that $\psi_{V_i}(\~v_i) = v_i$ and $\~g(\~v_2) = \~v_1$. 

Then by the argument in the proof of Lemma \ref{lemma local generator along closed fibre}, there exists an element $\rho \in R^{\prime}$ such that the ideal $(h_1 \circ \psi_{V_1})^*(\@I_T)_{\~v_1}$ in $\@O_{\~V, \~v_1}$ is generated by $\~f_1^*(\rho)$. Thus, we have 
\begin{eqnarray*}
(h_2 \circ \psi_{V_2})^*(\@I_T)_{\~v_2} & = & (h_1 \circ g \circ \psi_{V_2})^*(\@I_T)_{\~v_2} \\
& = & (h_1 \circ\psi_{V_1}\circ \~g)^*(\@I_T)_{\~v_2} \\
& = & \~g^*(\<\~f_1^*(\rho)\>) \\
& = & \<\~f_2^*(\rho)\> \text{.}
\end{eqnarray*}
However, by assumption, $(h_2 \circ \psi_{V_2})^*(\@I_T)_{\~v_2} = \psi_{V_2}^*(\<f_2^*(r)\>) = \~f_2^*(\psi_U^*(r))$. This proves that $\rho$ is a unit multiple of $\psi_{U}^*(r)$ in $R^{\prime}$. As in the proof of Lemma \ref{lemma local generator along closed fibre}, this shows that $h_1^*(\@I_T)_{v_1} = \<f_1^*(r)\>$.
\end{proof}

\begin{proposition}
\label{proposition local generator ghost homotopy}
Let $n \geq 0$ be an integer. Let $\alpha$ and $\alpha^{\prime}$ be sections of $X \to U$ which are connected by an $n$-ghost homotopy $\@H$. Suppose that $\@H$ (and hence $\alpha$, $\alpha'$) lifts to $\~X$. Then there exists $r \in R$ such that $\<r\> \supset I_0$ and $r$ generates $\alpha^*(\@I_T)=(\alpha')^*(\@I_T)$. Also, in this case $h_{\@H}^*(\@I_T)$ is generated by $f^*_{\@H}(r)$. 
\end{proposition}

\begin{proof}
This will be proved by induction on $n$. We begin with the case $n = 0$. Thus, suppose $\alpha$ and $\alpha^{\prime}$ are connected by a single $\#A^1$-homotopy $h: \#A^1_U \to X$, which lifts to $\~X$.  In other words, $h \circ \sigma_0 = \alpha$ and $h \circ \sigma_1 = \alpha'$, where $\sigma_i: U \to \A^1_U$ denotes the $i$-section, for $i=0,1$.  Let $r$ be such that 
%$\alpha \in \@A^{r_0}_r$, i.e. 
$\alpha^*(\@I_T) = \<r\>$. Then by Lemma \ref{lemma extending local generator from closed subscheme} (applied to $V_1 = \A^1_U$, $V_2=U$, $f_1 = f = pr_2$, where $pr_2$ is the projection $\#A^1_U \to U$, $f_2 = id_U$ and $g= \sigma_0$), we see that $h^*(\@I_T)|_{\sigma_0(u)}$ is generated by $pr_2^*(r)$. By Lemma \ref{lemma local generator along closed fibre}, there exists an open subscheme $W_0 \subset \#A^1_U$ containing the closed fiber $\#A^1_k \subset \#A^1_U$ such that $h^*(\@I_T)|_{W_0}$ is generated by $pr_2^*(r)$. In particular, $h^*(\@I_T)_{(\sigma_1(u))}$ is generated by $pr_2^*(r)$. Thus $(\alpha^{\prime})^*(\@I_T) = \sigma_1^*(h(\@I_T))$ is generated by $\sigma_1^*(pr_2^*(r)) = r$.  This completes the proof in the case $n = 0$. 

Now suppose the result has been proved for $m$-ghost homotopies where $m < n$. Suppose $\alpha$ and $\alpha^{\prime}$ are connected by an $n$-ghost homotopy 
\[
\@H:=(V \to \#A^1_U, W \to V \times_{\#A^1_U} V, \~\sigma_0, \~\sigma_1, h, \@H^W ) \text{.}
\]
Let $r, r^{\prime} \in R$ be such that $\alpha^*(\@I_T) = \<r\>$ and $(\alpha^{\prime})^*(\@I_T) = \<r^{\prime}\>$. Then, by Lemma \ref{lemma extending local generator from closed subscheme}, we see that $h_{\@H}^*(\@I_T)_{\~{\sigma_0}(u)} = \<f_{\@H}(r)\>$ and $h_{\@H}^*(\@I_T)_{\~{\sigma_1}(u)} = \<f_{\@H}(r^{\prime})\>$. Let $Z_0$ and $Z_1$ be the irreducible components of the closed fiber of $f_{\@H}^{-1}(u)$ of $V$ containing $\~{\sigma_0}(u)$ and $\~{\sigma_1}(u)$ respectively. By Lemma \ref{lemma local generator along closed fibre}, there exist open subschemes $W_i \supset Z_i$ of $V$ for $i = 0,1$ such that $h_{\@H}^*(\@I_T)|_{W_0}$ is generated by $f_{\@H}^*(r)$ and $h_{\@H}^*(\@I_T)|_{W_1}$ is generated by $f_{\@H}^*(r^{\prime})$. The morphisms $Z_i \to \#A^1_k$ are \'etale and thus there exists a point $z \in \#A^1_k \subset \#A^1_U$ lying in the image of both $Z_i \to \#A^1_k$ for $i = 0,1$. Thus, there exist points $z_i \in Z_i$ which map to $z$ under the morphism $V \to \#A^1_U$. Thus, we obtain the point $(z_1, z_2) \in V \times_{\#A^1_U} V$ and there exists a point $z_3 \in W$ which maps to $(z_1,z_2)$ under the morphism $W \to V \times_{A^1_U} V$. Let $g_1, g_2$ denote the two compositions
\[
W \to V \times_{\#A^1_U} V \stackrel{pr_i}{\to} V 
\]
where $pr_i: V \times_{\#A^1_U} V \to V$ is the projection on the $i$-th factor for $i = 1,2$. Thus $g_1^*(h_{\@H}^*(\@I_T))_{z_3}$ is generated by $g_1^*(f_{\@H}^*(r))$ and $g_2^*(h_{\@H}^*(\@I_T))_{z_3}$ is generated by $g_2^*(f_{\@H}^*(r^{\prime}))$. Since the morphisms $h_{\@H} \circ g_1$ and $h_{\@H} \circ g_2$ are $(n-1)$-ghost chain homotopic, the induction hypothesis implies that the ideals $\<g_1^*f_{\@H}^*(r)\>$ and $\<g_2^*f_{\@H}^*(r^{\prime})\>$ of $\@O_{W,z_3}$ are equal. It is easily seen that the morphisms $f_{\@H} \circ g_1$ and $f_{\@H} \circ g_2$ from $W \to U$ are identical. Since this is a smooth morphism, it is faithful and thus the equality $\<g_1^*f_{\@H}^*(r)\> = \<g_2^*f_{\@H}^*(r^{\prime})\>$ implies that $\<r\> = \<r^{\prime}\>$ in $R$ as desired. 

Now we need to prove that the ideal sheaf $h^*_{\@H}(\@I_T)$ is generated by $f^*_{\@H}(r)$. The above arguments show that there exists an open subscheme $V_0 \subset V$ containing the closed fiber $f^{-1}_{\@H}(u)$ such that $h_{\@H}^*(\@I_T)|_{V_0}$ is generated by $f^*_{\@H}(r)$. Of course, it is possible that $V \backslash V_0$ is non-empty. Suppose $v \in V \backslash V_0$ and let $f_{\@H}(v) = z \in U$. We define $U_z = \Spec \@O_{U,z}^h$. Applying the above arguments for $U_z$ instead of $U$, we see that there exists an element $\~r \in \@O_{U,z}^h$ such that $h_{\@H}(\@I_T)_{v^{\prime}}$ is generated by $f_{\@H}^*(\~r)$ for every $v^{\prime}$ in $f_{\@H}^{-1}(z)$ (in particular for $v^{\prime} = v$). But since $f_{\@H}^{-1}(z) \cap V_0$ is non-empty, we see that $\<\~r\> = r\@O_{U,z}^h$. Thus we see that $h^*_{\@H}(\@I_T)|_V$ is generated by $f_{\@H}^*(r)$. Now, by Lemma \ref{lemma extending local generator from closed subscheme}, it easily follows that $h^*_{\@H}(\@I_T)|_{\Sp(\@H_i)}$ is also generated by $f_{\@H}^*(r)$ for every $(n-1)$-ghost homotopy appearing in $\@H_W$. This completes the proof. 
\end{proof}

\begin{remark}
\label{remark local generator ghost homotopy}
To clarify the geometric intuition behind Proposition \ref{proposition local generator ghost homotopy}, we note a simple consequence. Observe that $\alpha^*(\@I_T)$ is the unit ideal if and only if $\alpha$ maps the closed point of $U$ into $T$. Thus, we see that $\alpha$ maps the closed point of $U$ into $T$ if and only if $\alpha'$ does so too. Now, if $u'$ is any other point of $U$, we may apply this argument to $\Spec(\@O_{U,u'})$ and conclude that $\alpha$ maps $u'$ into $T$ if and only if $\alpha'$ does so too. As far as $\~X$ is concerned, the above proposition says that an $n$-ghost homotopy of $U$ on $\~X$ can connect the lifts of $\alpha_1$ and $\alpha_2$ only if the schemes $\alpha_1^{-1}(T)$ and $\alpha_2^{-1}(T)$ are the same. Of course, this is only a necessary condition and may not be sufficient for $\alpha_1$ and $\alpha_2$ to be $n$-ghost homotopic. 
\end{remark}

\subsection{Applications to birationally ruled surfaces}
\label{subsection homotopies surfaces}

In this subsection, we will apply the results of Section \ref{subsection general result} to the context of birationally ruled smooth surfaces. 

\begin{lemma}
\label{lemma topological ruled surfaces}
Let $C$ be an $\#A^1$-rigid, essentially smooth, irreducible scheme of dimension $1$ over $k$. Let $\pi_0: E \to C$ be a $\#P^1$-bundle and let $\phi: X \to E$ be a birational morphism. Let $U$ be any essentially smooth, irreducible scheme over $k$ and let $\gamma: U \to C$ be a morphism which maps the generic point of $U$ to the generic point of $c$. Let $\alpha_1, \alpha_2: U \to X$ be two $\gamma$-morphisms which are connected by an $n$-ghost homotopy $\@H$ for some $n \geq 0$. Let $c$ be a closed point of $C$ and let $F$ be a component of $\pi^{-1}(c)$. Then, for any point $u \in \gamma^{-1}(c)$, we have the following:
\begin{itemize}
\item[(1)] $\alpha_1(u) \in F$ if and only if $f_{\@H}$ maps the fiber $\Sp(\@H)_u:= \Sp(\@H) \times_{U,u} \Spec(\kappa(u))$ (viewed as a closed subscheme of $\Sp(\@H)$) into $F$. 
\item[(2)] $\alpha_1(u) \notin F$ if and only if $f_{\@H}(\Sp(\@H)_u) \cap F = \emptyset$.
\end{itemize}
\end{lemma}

\begin{proof} Standard results about the birational geometry of smooth surfaces imply that $\phi$ can be written as a composition 
\[
X = X_m \stackrel{\pi_{m}}{\to} X_{m-1} \stackrel{\pi_{m-1}}{\to} \cdots \stackrel{\pi_1}{\to} X_0 = E
\]
where each $\pi_i: X_i \to X_{i-1}$ is a blowup at a smooth closed point of $X_{i-1}$. Let $\pi:X \to C$ denote the composition $\pi_0 \circ \phi$. If $c$ is any closed point of $C$, then the following statements are easily proved by induction on $m$:
\begin{itemize}
\item[(1)] $\pi^{-1}(c)$ is a connected scheme, each component of which is isomorphic to $\#P^1_k$. 
\item[(2)] If $\pi^{-1}(c)$ has more than one component, then for any component $F$ of $\pi^{-1}(c)$, the intersection number $F \cdot F$ is negative. On the other hand, if $\pi^{-1}(c)$ is irreducible, then $[\pi^{-1}(c)] \cdot [\pi^{-1}(c)] = 0$, where $[\pi^{-1}(C)]$ is the cycle associated to the subscheme $\pi^{-1}(c)$.  
\end{itemize}

If $m = 0$, there is nothing to prove. Assume that $m \geq 1$. As we noted above, the intersection number $F \cdot F$ is negative.  Hence, by Artin's contractibility criterion \cite[Theorem  2.3]{Artin}, there exists a (possibly singular) surface $X'$ and a morphism $p: X \to X'$ which maps $F$ to a single closed point of $X'$. Thus, $X$ can be obtained from $X'$ by blowing up closed subscheme supported at a closed point. 

Without any loss of generality, the scheme $U$ can be replaced by its henselization at the point $u$. Thus, we may assume that $U = \Spec(R)$ where $R$ is a henselian local ring and that $u$ is the closed point of $U$. 

Let $X_{\gamma} = X \times_{\pi,C,\gamma} U$ and let $E_{\gamma} = E \times_{\pi_0, C, \gamma} U$. Let $S$ denote the closed subscheme of $X'$ such that $X$ is obtained from $X'$ by blowing up $S$. Then $X_{\gamma}$ is obtained from $X'_{\gamma}$ by blowing up the scheme $T = S \times_{C,\gamma} U$. For $i = 1,2$, let $\beta_i: U \to X'_{\gamma}$ denote the morphism induced by the map $p \circ \alpha_i: U \to X'$. The homotopy $\@H$ induces an $n$-ghost homotopy $\~{\@H}$ connecting $\beta_1$ and $\beta_2$. By Proposition \ref{proposition local generator ghost homotopy}, there exists an element $r \in R$ such that the ideal sheaves $\beta_i^*(\@I_{T}) = (p \circ \alpha_i)^*(\@I_{S})$ and $f_{\~{\@H}}^*(\@I_{T}) = f_{\@H}^*(\@I_S)$ are generated by $r$.  If $\alpha_1(u) \in F$, then $r$ is a non-unit. Thus, the restriction of the ideal sheaf $f_{\@H}^*(\@I_T)$ to any point $u'$ of $\Sp(\@H)_u$ is not a unit ideal, which implies that $f_{\@H}$ maps $u'$ into $F$. 

On the other hand, if $\alpha_1(u) \notin F$, then $r$ is a unit. Thus, the restriction of the ideal sheaf $f_{\@H}^*(\@I_T)$ to any point $u'$ of $\Sp(\@H)_u$ is not a unit ideal, which implies that $f_{\@H}(u') \notin F$.
\end{proof}

\begin{remark}
\label{remark homotopies on ruled surfaces}
Suppose that $Z$ is the union of all lines which do not contain $\alpha_1(u)$. Then it follows from the above result that the $n$-ghost homotopy $\@H$ factors through $X \backslash Z$. 

In fact, we can say a little more. First, we observe that any point of $\pi^{-1}(c)$ lies on at most two components of $\pi^{-1}(c)$. Thus, $\alpha_1(u)$ may lie on one or two components. If a point $\alpha_1(u)$ is the intersection of two components $F_1$ and $F_2$, then the entire fiber $\@Sp(\@H)_u$ is mapped to the point $\alpha_1(u)$. 

Suppose that $\alpha_1(u)$ only lies on one component of $\pi^{-1}(c)$, which we denote by $F$. As above, let $Z$ be the union of all the other components. Then $\Sp(\@H)_u$ is mapped into $F \backslash Z$. If $F \cap Z$ consists of at least two points, then $F \backslash Z$ is $\#A^1$-rigid. Since the restriction of $f_{\@H}$ to $\Sp(\@H)_u$ defines an $n$-ghost homotopy of $\Spec(\kappa(u))$, it follows that this $n$-ghost homotopy must be constant. In other words, in this special case too $\Sp(\@H)_u$ is mapped to the point $\alpha_1(u)$. However, we note that the situation is quite different if $F \cap Z$ consists of only one point.   
\end{remark}

\section{Geometry of ruled surfaces}
\label{section geometry of ruled surfaces}

Our proof of the main theorem (Theorem \ref{theorem ruled iterations of S}) is based on the classification of smooth projective surfaces.  We will therefore need a formalism to handle blowups of points on a minimal ruled surface.  We will briefly review some elementary constructions on schemes in Section \ref{subsection constructions}, all of which basically follow from resolution of indeterminacy and universal property of blowups.  The material in Section \ref{subsection constructions} should be quite obvious to the expert, but we present the proofs for the sake of completeness and since we use nonstandard notation to facilitate the book-keeping needed for our purposes.  Section \ref{subsection nodal blowups of ruled surfaces} gives a systematic way of handling a special class of blowups of a minimal ruled surface and Section \ref{subsection etale cover} contains a key result about local geometry of such blowups (Theorem \ref{theorem etale cover}), which plays a crucial role in our proof of Theorem \ref{theorem ruled iterations of S}. 

\subsection{Review of some elementary constructions on schemes}
\label{subsection constructions}

We fix a base scheme $\Spec R$, where $R$ is a noetherian domain.  All the $R$-schemes considered in this subsection will be separated, integral and of finite type over $\Spec R$.  We comment that some of the restrictions we have placed on the base ring $R$ as well as the $R$-schemes under consideration may not be necessary, but they allow us to write simpler proofs and are sufficient for our purposes.

Let $X$ be an integral, separated $R$-scheme and let $f: X \dashrightarrow \#P^1$ be a rational function.  We will construct three kinds of $R$-schemes $Z$ along with structure morphisms $Z \to X$ such that the pullback of $f$ to $Z$ has certain special properties.

\subsubsection{Resolving indeterminacies of a rational function}

Let $T_0, T_1$ denote homogeneous coordinates on $\#P^1_R$, so that $\#P^1_R = \Proj R[T_0, T_1]$. Recall that for any $R$-scheme $Z$ and a morphism $\pi: Z \to \#P^1_R$, we obtain an invertible sheaf $\@L := \pi^*(\@O(1))$ on $Z$ and a pair of global sections $s:= \pi^*(T_0)$ and $t:= \pi^*(T_1)$ which generate $\@L$.  Conversely, given a line bundle $\@L$ on an $R$-scheme $Z$ and an ordered pair of global sections $s,t$ generating $\@L$, one can construct a morphism $\pi: Z \to \#P^1_R$ such that there exists an isomorphism $\pi^*(\@O(1)) \to \@L$ which maps $\pi^*(T_0)$ to $s$ and $\pi^*(T_1)$ to $t$. 

Suppose we are given a line bundle $\@L$ on an $R$-scheme $X$ and an ordered pair of global sections $(s,t)$ such that $s$ and $t$ do not generate $\@L$. Then the set of all points $x \in X$ where $s_x$ and $t_x$ generate $\@L_x$ is an open subset $U$ of $X$.  If $U$ is nonempty (and hence, dense in $X$), we obtain a rational function on $X$.  We would like to know when this rational function can be extended to a morphism from $X$ to $\#P^1_R$. 

Recall that on an integral scheme $X$, any invertible sheaf is isomorphic to a subsheaf of the constant sheaf $\@K$ of rational functions on $X$.  Thus, for any invertible sheaf $\@L$ on $X$, every global section of $\@L$ may be considered as an element of $\@K(X)$. In particular, given two global sections $s,t$ of $\@L$, we will speak of the ratio $s/t$ as being an element of $\@K(X)$.

\begin{lemma}
\label{lemma extending rational functions}
Let $X$ be a separated, integral $R$-scheme. Let $\@L$ be an invertible sheaf on $X$ and let $s,t$ be global sections defining a rational function on $X$. Then, $f$ can be extended to a morphism $X \to \#P^1_R$ if and only if the sections $s$ and $t$ generate an invertible subsheaf of $\@L$. 
\end{lemma}

\begin{proof}
Suppose we are given the rational function $f$ as above. Clearly, $s$ and $t$ are not both equal to the zero section. Suppose that $t \neq 0$. Then $f$ is clearly defined on the open subset $U$, consisting of points $x \in X$ such that $s_x$ and $t_x$ generate $\@L_x$. Let $\@M$ be the subsheaf of $\@L$ generated by $s$ and $t$. Note that $\@M|_{U} = \@L|_{U}$. 

Suppose that the sheaf $\@M$ is locally principal. Then, as the sections $s,t$ generate $\@M$, by the above comments we obtain a morphism $\~f: X \to \#P^1_R$ such that $\~f^*(\@O(1))$ is isomorphic to $\@M$ via a morphism that maps $\~f^*(T_0)$ to $s$ and $\~f^*(T_1)$ to $t$. Clearly, $\~f|_{U} = f|_{U}$.  

Conversely, suppose that $f$ can be extended to a morphism $\~f: X \to \#P^1$. Thus, there exists an invertible sheaf $\@N$ on $X$ with global sections $u,v$ which generate $\@N$ and an isomorphism $\~f^*(\@O(1)) \to \@N$ which maps $f^*(T_0)$ to $u$ and $f^*(T_1)$ to $v$. On the set $U$ defined above, this morphism agrees with $f$ and so we must have $s/t = u/v$ as elements of $\@K$. As $u$ and $v$ generate an invertible sheaf on $X$, at any point $x \in X$, either $u/v \in \@O_{X,x}$ or $v/u \in \@O_{X,x}$. Thus we also have $s/t \in \@O_{X,x}$ or $t/s \in \@O_{X,x}$ for every $x \in X$. This shows that the sheaf $\@M$ is locally principal. 
\end{proof}

\begin{proposition}
\label{proposition resolving indeterminacies}
Let $X$ be an integral scheme over $R$. Let $f$ be a rational function on $X$. There exists a pair $(X[f], \pi_f)$ consisting of a scheme $X[f]$ and a morphism $\pi_f: X[f] \to X$ satisfying the following properties:
\begin{enumerate}[label=$(\alph*)$]
\item The pullback of $f$ to $X[f]$ can be extended to a morphism $\~f: X[f] \to \#P^1_R$. 
\item Given any $R$-scheme $Y$ and a morphism $\phi: Y \to X$ such that $\phi^*(f)$ extends to a morphism $Y \to \#P^1_R$, there exists a unique map $\~{\phi}: Y \to X[f]$ such that $\phi = \pi_f \circ \~{\phi}$. 
\end{enumerate}
\end{proposition} 

\begin{proof}
The statement of the theorem is clearly local with respect to $X$ and so it will suffice to prove the result when $X$ is affine. Thus, we may assume that $f$ is given by an ordered pair $(s,t)$ of global sections of the trivial bundle (so that $f = s/t$ as an element of the function field of $X$). Let $\@I$ be the ideal subsheaf of $\@O_X$ generated by $s,t$. Let $\pi_f: X[f] \to X$ be the blowup of $X$ at the ideal sheaf $\@I$. Then, the ideal sheaf $\pi_f^{-1}(\@I) \cdot \@O_{X[f]}$ is invertible. Thus, by Lemma \ref{lemma extending rational functions}, the rational function $\pi_f^*(f)$ extends to a morphism $\~f: X[f] \to \#P^1_R$. (This construction is given in \cite[Chapter II, Example 7.17.3]{Hartshorne}.)

Given any pair $(Y,\phi)$ as in property (b) in the statement of the theorem, we see by Lemma \ref{lemma extending rational functions} that the ideal sheaf $\phi^{-1}(\@I) \cdot \@O_{X[f]}$ is invertible. Thus, we obtain the result  by applying the universal property of blowups (see \cite[Chapter II, Prop. 7.14]{Hartshorne}). 
\end{proof}

Note that even if $f$ and $g$ both induce morphisms from $X \to \#P^1_R$, the product $fg$ may not do so. However, the following lemma is easy to prove:
\begin{lemma}
Let $X$ be a separated, integral $R$-scheme. Let $f \in \@O^{\times}_X(X)$ and let $g$ be a rational function on $X$. Then:
\begin{enumerate}[label=$(\alph*)$]
\item $g$ induces a morphism from $X$ to $\#P^1_R$ if and only if $fg$ induces a morphism from $X$ to $\#P^1_R$. 
\item $X[g] \simeq X[fg]$ as $X$-schemes. 
\end{enumerate}
\end{lemma}

\begin{notation} Let $f_1, \ldots, f_n$ be rational functions on $X$. We will write $X[f_1, \ldots f_n]$ instead of $X[f_1][f_2] \cdots [f_n]$. 
\end{notation}

\subsubsection{Attaching an $n$-th root of a rational function}

Given a regular function $f: X \to \#P^1_R$ and an integer $n>1$, we would like to construct the final object in the category of $X$-schemes on which $f$ has an $n$-th root. We need to be careful about the fact that the $n$-th root may not be unique. So in order to conveniently phrase the universal property of our construction, we also pick out a specific $n$-th root of (the pullback of) $f$. 

\begin{proposition}
\label{proposition attaching roots}
Let $X$ be a separated, integral $R$-scheme and let $f: X \to \#P^1_R$ be a morphism. Let $n>0$ be an integer. Then there exists a triple  $$(X[f^{1/n}],\pi_{f,n}, f^{1/n})$$ consisting of a scheme $X[f^{1/n}]$, a morphism $\pi_{f,n}: X[f^{1/n}] \to X$ and a morphism $f^{1/n}: X[f^{1/n}] \to \#P^1_R$ such that the following conditions hold: 
\begin{enumerate}[label=$(\alph*)$]
\item The morphism $f^{1/n}$ is an $n$-th root of  $\pi^*_{f,n}(f)$. %(wherever the latter is defined) in the obvious sense. 
We will call $f^{1/n}$ the \emph{structural $n$-th root} of the triple $(X[f^{1/n}],\pi_{f,n}, f^{1/n})$. 
\item Given any triple $(Y,\phi, g)$ where $Y$ is an $R$-scheme, $\phi: Y \to X$ is a morphism and $g: Y \to \#P^1_R$ is a morphism which is an $n$-th root of $\phi^*(f)$, then there exists a unique morphism $\~{\phi}: Y \to X[f^{1/n}]$ such that $\phi = \pi_{f,n} \circ \~{\phi}$ and $\~{\phi}^*(f^{1/n}) = g$.
\end{enumerate}
\end{proposition}

\begin{proof}
As above, let $T_0, T_1$ be the homogeneous coordinates on $\#P^1_R$. There exists a line bundle $\@L$ on $X$ and global sections $s,t$ of $\@L$ such that $f^*(\@O(1)) = \@L$ via an isomorphism that maps $f^*(T_0)$ to $s$ and $f^*(T_1)$ to $t$. Then, as elements of $\@K(X)$, we have the equality $f = s/t$. 

Let $\phi_n: \#P^1_R \to \#P^1_R$ be the $n$-th power map given by the graded ring homomorphism 
\[
R[T_0, T_1] \to R[T_0, T_1]; \quad T_0 \mapsto T_0^n, T_1 \mapsto T_1^n. 
\]
We then define $X[f^{1/n}]$ to be the fiber product $X \times_{f, \#P^1_R, \phi_n} \#P^1_R$ with 
\[
p_1: X[f^{1/n}] = X \times_{f, \#P^1_R, \phi_n} \#P^1_R \to X, \hspace{1cm} p_2: X[f^{1/n}] = X \times_{f, \#P^1_R, \phi_n} \#P^1_R \to \#P^1_R
\]
being the projections on the first and second factors respectively. We define $\pi_{f,n}$ to be equal to $p_1$ and we define $g = p_2^*(T_0/T_1)$. It is easy to see that $g^n = f$. 

Now suppose that $(Y, \phi, g)$ is another triple as in property (2) in the statement of the theorem. Then, the universal property of the fiber product immediately gives required morphism $\~{\phi}$ from the morphisms $\phi: Y \to X$ and $g: Y \to \#P^1_R$. 
\end{proof}

We will simply write $X[f^{1/n}]$ for the triple $(X[f^{1/n}], \pi_{f,n}, f^{1/n})$, when the rest of the data is clear from the context.  Thus, for instance, suppose $p: Z \to X$ and $g: Z \to \#P^1_R$ are morphisms, we may say that $\phi: X[f^{1/n}] \to (Z, p, g)$ is an isomorphism to mean that $\phi$ is an isomorphism from the scheme $X[f^{1/n}]$ to the scheme $Z$ such that $\pi_{f,n}\circ \phi = p$ and $\phi^*(g) = f^{1/n}$.  The following lemma is obvious and so we omit its proof. 

\begin{lemma}
Let $X$ be a separated, integral $R$-scheme and let $f \in \@O^\times(X)$, viewed as a morphism from $X$ to $\#P^1_R$.  If $n \in R^{\times}$, then the morphism $X[f^{1/n}] \to X$ is \'etale. 
\end{lemma}

\subsubsection{Turning a regular function into a unit}

Suppose that $f$ defines a morphism from $X$ to $\#P^1_R$. Then, $X\{f\}$ will denote a scheme with a given morphism $\eta_{f}: X\{f\} \to X$ such that the following conditions hold:
\begin{enumerate}[label=$(\alph*)$]
\item $\eta_f^*(f)$ is in $\@O_{X\{f\}}^\times(X\{f\})$.
\item Given any scheme $Y$ and a morphism $\phi: Y \to X$ such that $\phi^*(f)$ is in $\@O_Y^\times(Y)$, there exists a unique map $\~{\phi}: Y \to X\{f\}$ such that $\phi = \pi_f \circ \~{\phi}$. 
\end{enumerate}

Clearly, $X\{f\}$ is just the open subscheme of $X[f]$ which is the complement of the support of $div(f)$. 

\begin{remark}
We observe that due to the universal properties of the three constructions 4.1.1-4.1.3, they may be permuted. In other words, if $f$ is a rational function and $g,h$ are regular functions on $X$ and $n>1$ is an integer, then we have isomorphisms $X[f]\{g\} \simeq X\{g\}[f]$, $X[f][g^{1/n}] \simeq X[g^{1/n}][f]$ and $X\{g\}[h^{1/n}] \simeq X[h^{1/n}]\{g\}$. However, for the second and third isomorphisms to make sense, we should also keep track of the structural $n$-th roots.
\end{remark}

\subsection{Nodal blowups of ruled surfaces}
\label{subsection nodal blowups of ruled surfaces}

Let $C$ be a smooth, $1$-dimensional scheme over $k$. In this section, we will work with schemes that are obtained by successive blowups of $\#P^1_C$ at smooth, closed points.  Thus, all such schemes will come equipped with canonical birational maps between them.  We will use the following conventions involving rational functions, points and curves on such schemes for the sake of brevity. 

\begin{conventions}
\hspace{1cm}
\begin{enumerate}
\item \emph{Rational functions:} Given any scheme $X$ obtained from $\#P^1_C$ by successive blowups, its function field will be canonically isomorphic to that of $\#P^1_C$. We will use this canonical isomorphism to identify the two function fields. Thus, the same symbol will be used to denote a rational function on $\#P^1_C$ and its pullback to $X$. In particular, a rational function in $x$ and $y$ with coefficients in $k$ can be interpreted as a rational function on any smooth scheme that is birational to $\#P^1_C$. 

\item \emph{Points:} Suppose $X$ and $Y$ are both birational to $\#P^1_C$ and let $\phi: X \dashrightarrow Y$ be the unique birational map such that the diagram
\[
\xymatrix{
X \ar@{-->}[rr]^{\phi} \ar@{-->}[rd] & & Y \ar@{-->}[ld] \\
 & \#P^1_C & 
}
\]
commutes. Let $P$ be a point on $X$ such that $\phi$ is an isomorphism on an open neighbourhood of $P$. Then the image $\phi(P)$ in $Y$ will also be denoted by the symbol $P$.

\item \emph{Curves:} Suppose $X$, $Y$ and $\phi$ are as in (b). If $B$ is a curve on $X$, the map $\phi$ is defined on an open subset $B'$ of $B$. The closure of $\phi(B')$, which is called the \emph{proper transform of $B$ under $\phi$}, will also be denoted by the symbol $B$.  
\end{enumerate}
\end{conventions}

We fix the following setting for the rest of the article.

\begin{notation}
\label{notation nodal blowups}
Fix an algebraically closed field $k$ of characteristic $0$.  

\begin{enumerate}
\item Let $x$ be a variable and let $A$ be the Henselization of the polynomial ring $k[x]$ at the maximal ideal $\<x\>$. Let $C = \Spec A$ and let $c_0$ denote the closed point of $C$.  Every $1$-dimensional regular henselian local ring containing $k$ is isomorphic to $A$.  

\item Let $Y$ and $Z$ be variables which denote the homogeneous coordinates on $\#P^1_C$. Thus, $\#P^1_C = \Proj A[Y,Z]$. 

\item Let $y$ denote the rational function $Y/Z$ on $\#P^1_C$ which is defined on the open subscheme of $\#P^1_C$ defined by the condition $Z \neq 0$. This open subscheme is simply $\Spec A[y] \simeq \#A^1_C$ and the point $(c_0, [0:1])$ is defined by the ideal $\<x,y\>$ of $A[y]$.  

\item Let $\ell_{\infty}$ denote the closed subscheme $C \times \{[0:1]\}$ of $\#P^1_C$, which is the divisor of zeros of $y$. The closed subscheme $C \times \{[1:0]\}$ is the divisor of poles of $y$ and will be denoted by $\ell_{-\infty}$. 
\end{enumerate}
\end{notation}

\begin{definition}
\label{definition nodal blowups}
Let $X$ be any scheme that is obtained from $\#P^1_C$ by a finite number (possibly zero) of successive blowups at smooth, closed points.
\begin{enumerate}
\item The fibre of $X \to C$  over $c_0$ is a connected scheme, the irreducible components of which are isomorphic to $\#P^1_k$. We refer to these as \emph{lines} on $X$.

\item We define a \emph{pseudo-line} to be any curve that is either a line on $X$ or the proper transform of $\ell_{\infty}$.

\item A \emph{node} on $X$ is defined to be the intersection point of two pseudo-lines. It is easy to see that any node is the point of intersection of exactly two pseudo-lines.  Thus, for instance, the point $(c_0, [0:1])$ is the only node on the scheme $\#P^1_C$. 
\end{enumerate}

We will denote by $\@N$ the collection of schemes $X$ admitting a morphism $X \to \#P^1_C$, which factors as  
\[
X = X_r \xrightarrow{\pi_r} X_{r-1} \xrightarrow{\pi_{r-1}} \cdots \xrightarrow{\pi_1} X_0 = \#P^1_C
\]
with $r \geq 0$, where for all $i\geq 1$, $\pi_i: X_i \to X_{i-1}$ is the blowup of $X_{i-1}$ at some \emph{node} of $X_{i-1}$.  Note that the definition of $\@N$ does not merely depend on the scheme $C$. It depends on the choices of the parameter $x$ on $C$ as well as the homogeneous coordinates $Y$ and $C$ on $\#P^1_C$.
\end{definition}

Since all the nodes in all the $X \in \@N$ lie over the point $(c_0, [0:1])$, we see that these ideals defining these nodes are generated by rational functions in $x$ and $y$. We will now describe all such ideals. 

\begin{example}
To begin with, we examine the blowup of $\#P^1_C$ at the point $(c_0, [0:1])$. As we noted above, this point is locally defined by the ideal $\<x,y\>$. It is the intersection of the unique line in $\#P^1_C$ and the pseudo-line $\ell_{\infty}$. The line in $\#P^1_C$ has $y$ as a parameter while the pseudo-line $\ell_{\infty}$ has $x$ as a parameter. When we blow up the point, the exceptional divisor is a line having parameter $x/y$. It meets the proper transform of the line on $\#P^1_C$ in a point defined by the ideal $\<x/y,y\>$. It meets the proper transform of $\ell_{\infty}$ in a point defined by the ideal $\<x,y/x\>$. 
\end{example}

More generally, suppose that a node, which is the intersection of pseudo-lines $C_1$ and $C_2$ is defined by the ideal $\<\alpha, \beta\>$. We also assume that $\alpha$ is a parameter on $C_1$ and $\beta$ is a parameter on $C_2$. When this point is blown up, the exceptional divisor is a line having parameter $\alpha/\beta$. The proper transforms of $C_1$ and $C_2$, which we continue to denote by the symbols $C_1$ and $C_2$ respectively, have parameters $\alpha$ and $\beta$ respectively. The exceptional divisor meets $C_1$ in a point defined by the ideal $\<\alpha, \beta/\alpha\>$ and it meets $C_2$ in a point defined by the ideal $\<\alpha/\beta, \beta\>$. Since we start with the ideal $\<x,y\>$ it is easy to see that all the nodes will be locally defined by ideals of the form $\<x^a/y^b, y^d/x^c\>$ where $a,b,c,d$ are non-negative integers. Also, at such a node, the rational function $x^a/y^b$ is a parameter on one of the pseudo-lines through the node, while $y^d/x^c$ is a parameter on the other pseudo-line. 

Thus, we see that for any $X \in \@N$, any line on $X$ has a parameter of the form $x^a/y^b$ for non-negative integers $a$ and $b$. Observe that the integers $a$ and $b$ are uniquely determined by this line. Indeed, if $x^a/y^b$ and $x^c/y^d$ are parameters on the same line, they are related by a fractional linear transformation. In other words, there exist $\alpha, \beta, \gamma, \delta$ in $k$ such that $\alpha \delta - \beta \gamma \neq 0$ such that 
\[
x^a/y^b = \frac{\alpha (x^c/y^d) + \beta}{\gamma(x^c/y^d) + \delta}. 
\]
As $x$ and $y$ are algebraically independent, it is easy to prove that this can only happen if $(a,b) = (c,d)$. 

\begin{lemma}
\label{lemma blowups description}
Let $X \in \@N$. Then:
\begin{itemize}
\item[(a)] If a line $l$ of $X$ has a parameter of the form $x^a/y^b$, then $a$ and $b$ are coprime non-negative integers.  
\item[(b)] If $C_1$ and $C_2$ are pseudo-lines on $X$ with parameters $x^a/y^b$ and $x^c/y^d$ which intersect at a node defined by the ideal $\<x^a/y^b, y^d/x^c\>$, then $ad-bc = 1$ (and so, in particular, $a/b > c/d$). 
\end{itemize}
\end{lemma}

\begin{proof}
We choose a sequence of blowups 
\[
X  = X_r \stackrel{\pi_{r}}{\to} X_{r-1} \stackrel{\pi_{r-1}}{\to} \cdots \stackrel{\pi_1}{\to} X_0 =\#P^1_C
\] 
where $\pi_i$ is the blowup of $X_{i-1}$ at a node.  

We prove these statements on all the $X_i$ by induction on $i$. They are clearly both true for $i = 0$. Suppose we know these statements to be true for all the nodes on $X_i$, where $i\geq 0$ and that $X_{i+1}$ is obtained from $X_i$ by blowing up a node given by the maximal ideal $\<x^a/y^b, y^d/x^c\>$. By the induction hypothesis, $a/b>c/d$ and $ad-bc = 1$. Then, on $X_{i+1}$, we have two new nodes given by the maximal ideals $\<x^{a+c}/y^{b+d}, y^d/x^c\>$ and $\<x^a/y^b, y^{b+d}/x^{a+c}\>$.  Since $ad-bc = 1$, we also get $(a+c)d - (b+d)c = ad-bc = 1$ and $a(b+d) - b(a+c) = ad-bc = 1$. Clearly, this implies that $a+b$ and $c+d$ are coprime. This completes the proof. 
\end{proof}

This shows that if two intersecting lines on $X$ have parameters $x^a/y^b$ and $x^c/y^d$, then the fractions $a/b$ and $c/d$ are in reduced form and that they are adjacent to each other on the Stern-Brocot tree (see \cite[Section 4.5]{GKP}). Indeed, we see that if $X$ in $\@N$ has $n$ lines, we can associate to it the sequence sequence of rational numbers $r_0 = -\infty < r_1 = 0 < \ldots < r_n = 1 < r_{n+1} = \infty$ such that:
\begin{itemize}
\item[(1)] For $1 \leq i \leq n$, if $r_i$ is written in reduced fractional form as $a_i/b_i$ for some non-negative integers $a_i$, $b_i$, then $X$ contains a line parameterized by $x^{a_i}/y^{b_i}$. We will label this line as $\ell_{r_i}$. For $i = 0$ and $i=n+1$, the symbol $\ell_{r_i}$ will denote the pseudo-lines $\ell_{-\infty}$ and $\ell_{\infty}$ respectively.   
\item[(2)] For $1 \leq i \leq n$, the line $\ell_{r_i}$ only meets the pseudo-lines $\ell_{r_{i-1}}$ and $\ell_{r_{i+1}}$. If $r_i$ and $r_{i+1}$ are written in reduced fractional form as $r_i = a_i/b_i$ and $r_{i+1} = a_{i+1}/b_{i+1}$ then they meet in a node defined by the maximal ideal $\<x^{a_{i+1}}/y^{b_{i+1}}, y^{b_i}/x^{a_i}\>$. (Here, if $r_{i+1} = \infty$, we choose $a_{i+1} = 1$ and $b_{i+1} = 0$.)
\end{itemize}
Note that this also shows that for pair of coprime integers $(a,b)$ with $a<b$, there is at most one line on $X$ with parameter $x^a/y^b$. Thus, labelling such a line as $\ell_{r}$ does not cause any conflicts. It is also easy to see if such lines exist on two schemes $X_1$, $X_2$ in $\@N$, then the canonical birational maps from $X_1$ to $X_2$ will take the lines into each other. Thus, our labelling respects the convention mentioned above, regarding using the same symbols for lines that are proper transforms of each other. 

\begin{lemma}
\label{lemma zeros and poles}
Let $0 < r \leq 1$ be a rational number, represented in reduced rational form as $a/b$. Let $X \in \@N$ be such that it contains a line labelled $\ell_r$. Then the rational function $x^a/y^b$ defines a morphism from $X$ to $\#P^1_R$. The support of the divisor of zeros of $x^a/y^b$ is given by 
\[
\Supp(div_0(x^a/y^b)) = \ell_{-\infty} \cup \left( \bigcup_{s<r} \ell_s \right)
\]
and the support of the divisor of poles of $x^a/y^b$ is given by
\[
\Supp(div_{\infty}(x^a/y^b)) = \ell_{\infty} \cup \left(\bigcup_{s>r} \ell_s \right).  
\]
(Here the unions are over all $s$ such that $X$ contains a line labelled $\ell_s$.) 
\end{lemma}

\begin{proof}
We omit the proof of this lemma, but it can be easily proved using an inductive argument as in the proof of Lemma \ref{lemma blowups description}. 
\end{proof}

The following lemma, which says that any ruled surface can be obtained from a suitable blowup of a nodal blowup, will be used in Section \ref{section general case}. 

\begin{lemma}
\label{lemma single point support}
Let $\pi:X \to \#P^1_C$ be a composition of successive blowups at smooth closed points. Then, there exist morphisms
\[
X \stackrel{\phi}{\to} X' \stackrel{\psi}{\to} \#P^1_C
\]
such that the following conditions hold:
\begin{enumerate}[label=$(\alph*)$]
\item $\phi$ and $\psi$ are compositions of successive blowups at smooth closed points. 
\item $X'$ is an element of $\@N$.
\item $\phi:X \to X'$ is a blowup whose support does not contain any of the nodes of $X'$.
\end{enumerate}
\end{lemma}
\begin{proof}
The morphism $\pi$ can be viewed as the blowup of a closed subscheme of $\#P^1_C$ corresponding to an ideal sheaf $\@J$ which is supported on a codimension $2$ subscheme of $\#P^1_C$. We will first reduce to the situation where the $\Supp(\@J) = \{(c_0,[0:1])\}$. 

If $\Supp(\@J)$ consists of a single point, by a suitable linear change of coordinates, we may assume that this point is actually $(c_0, [0:1])$. Let $N$ denote the number of components of $\pi^{-1}(\ell_0 \backslash \{(c_0, [0:1])\}$. If $N =1$, then $\Supp(\@J) = \{(c_0, [0:1])\}$ as desired. If $N>1$, then there exists at least one other point of $\ell_0$ which is in $\Supp(\@J)$. By a linear change of coordinates which leaves $(c_0, [0:1])$ fixed, we may ensure that the point $(c_0, [1:0])$ is in $\Supp(\@J)$. Let $\~X \to \#P^1_C$ denote the blowup at $(c_0, [1:0])$. Clearly, $\pi: X \to \#P^1_C$ factors uniquely through $\~X$. Let $\~{\pi}: \~X \to \#P^1_C$ be the morphism which contracts the line $\ell_0$ on $\~X$. Let $\pi': X \to \#P^1_C$ be the composition $X \to \~X \stackrel{\~{\pi}}{\to} \#P^1_C$. It is easy to see that the number of components of $(\pi')^{-1}(\ell_0 \backslash \{(c_0, [0:1])\})$ is strictly less than $1$. Repeating this process with $\pi'$ in place of $\pi$ if necessary, we eventually obtain a morphism $X \to \#P^1_C$ which is a blowup at a subscheme supported at the point $(c_0, [0:1])$. So, from this point onward, we assume that $\Supp(\@J) \cap \ell_0 = \{(c_0, [0:1])\}$. 

The morphism $X \to \#P^1_C$ can be written as a composition 
\begin{equation}
\label{equation blowup factorization}
X := X_m \to X_{m-1} \to \cdots \to X_1 \to X_0:= \#P^1_C
\end{equation}
where $\pi_{i,i-1}: X_i \to X_{i-1}$ is the blowup at a smooth closed point $Q_{i-1}$ of $X_{i-1}$. For any $i > j$, let $\pi_{i,j}$ denote the composition $\pi_{j+1,j} \circ \cdots \circ \pi_{i,i-1}$ and for $i = j$, let $\pi_{i,i}$ denote the identity morphism. 

For any integer $i$, $0 \leq i \leq m$, we call a point $Q \in X_i$ a \emph{pure node} if for any $1 \leq j \leq i$, the point $\pi_{i,j}(Q)$ is a node (see Notation \ref{notation nodal blowups}). We observe that $Q_0$ is a pure node.  We first claim that the above factorization of the morphism $X \to \#P^1_C$ can be chosen so that there exists an integer $0 < m' \leq m$ such that $Q_r \in X_r$ is a pure node if and only if $r < m'$. We will do this by modifying the above sequence of blowups repeatedly. 

If $Q_r$ is not a pure node for any $r \geq 1$, we may take $m' = 1$. So, we assume the contrary and let $p$ be the largest non-negative integer such that $Q_r$ is a pure node for every $r < p$, but $Q_p$ is not a pure node. If $Q_r$ is not a pure node for any $r>p$, we may simply take $m' = p$. If not, let $q$ be the smallest integer greater than $p$ such that $Q_q$ is a pure node. Thus, $\pi_{q,q-1}(Q_p) \neq Q_{q-1}$. Let $X'_q$ be the blowup of $X_{q-1}$ at $Q'_{q-1}:= \pi_{q,q-1}(Q_q)$. Let $Q'_{q} \in X'_q$ be the unique point of $X'_q$ which lies over $Q_{q-1}$. If $X'_{q+1}$ is the blowup of $X'_q$ at $Q'_{q}$, it is clear that we have a canonical isomorphism $X_{q+1} \to X'_{q+1}$ such that the diagram
\[
\xymatrix{
X'_{q+1} \ar[r]^{\cong}\ar[d] & X_{q+1} \ar[d] \\
X'_q \ar[r] \ar[dr] & X_{q} \ar[d] \\
& X_{q-1}
}
\]
commutes. Thus, we may now relabel $Q'_{q-1}$ as $Q_{q-1}$, $Q'_q$ by $Q_q$ and $X'_q$ by $X_{q}$ to get a different factorization of the morphism $X \to \#P^1_C$ as a sequence of blowups. However, now the smallest integer $r>p$ such that $Q_r$ is a pure node is $q-1$. If $q-1 > p$, we may repeat this process. Continuing in this manner, we modify the sequence of blowups until we come to a situation where $Q_r$ is a pure node for all $r < p+1$. 
We perform this entire procedure repeatedly until we obtain a sequence with the required property.  

We then set $X' := X_{m'}$ %and $m'' := m - m'$. 
and obtain the desired factorization 
\[
X \stackrel{\phi}{\to} X' \stackrel{\psi}{\to} \#P^1_C
\]
of $\pi$, clearly satisfying the conditions (a), (b) and (c). 
\end{proof}

\subsection{Local geometry of ruled surfaces}
\label{subsection etale cover}

Our proof of the main theorem (Theorem \ref{theorem ruled iterations of S}(b)) will be achieved by induction on the number of blow-ups required to obtain a given non-minimal ruled surface.  In this subsection, we prove a result regarding the local geometry of nodal blowups of $\P^1_C$, which will play a crucial role in the induction argument.  We keep the notation and conventions from Section \ref{subsection nodal blowups of ruled surfaces}.

\begin{definition}

\label{definition U_r}
Let $X \in \@N$ (see Definition \ref{definition nodal blowups}) and let $X_{c_0}$ denote the fiber of $X \to C$ over $c_0$. 
\begin{enumerate}
\item Let $r$ be a non-negative rational number, written in reduced form as $r = a/b$, $a,b \in \#Z$.  Choose a scheme $X \in \@N$ which has a line labelled $\ell_r$. We define the open subset $U_r$ of $X$ by 
\[
U_r := X \backslash (\Supp(div(x^a/y^b)) \cap X_{c_0}).
\]
If $r$ happens to be a positive integer, choose a scheme $X \in \@N$ which has a line labelled $\ell_r$ and no lines labelled $\ell_s$ for any $s>r$. We define the open subset $\~U_r$ of $X$ by 
\[
\~U_r := X \backslash (\Supp(div_0(x^1/y^r) \cap X_{c_0})
\]
We also define 
\[
\~U_0:= \#P^1_C \backslash \{(c_0, [0:1])\}.
\]
\item Let $r$ and $s$ be two non-negative rational numbers such if they are written in reduced form as $r = a/b$ and $s = c/d$, then $r-s = 1/(bd)$. Then, one can show that there exists an element $X$ of $\@N$ on which there exist lines $\ell_r$ and $\ell_s$, which meet in a point. We define the open subset $U_{r,s}$ of $X$ by
\[
U_{r,s} := X \backslash \left[\left(\Supp(div_{\infty}(x^a/y^b)) \cup \Supp(div_0(x^c/y^d)) \right) \cap X_{c_0} \right].
\]
If $r$ happens to be a positive integer, one can show that there exists an element $X$ of $\@N$ on which there exist lines $\ell_r$ and $\ell_s$ which meet in a point, and such that there does not exist any line labelled $\ell_t$ for $t>r$.  We define the open subset $\~U_{r,s}$ of $X$ by 
\[
\~U_{r,s} := X \backslash (\Supp(div_0(x^c/y^d))  \cap X_{c_0}).
\]
If $r = 0$ and $s$ is of the form $1/d$ for some integer $d$, we define 
\[
\~U_{0,s} := X \backslash (\Supp(y^1/x^d) \cap X_{c_0} ) \text{.}
\]
\end{enumerate}
It can be easily verified that the isomorphism classes of $U_r$, $\~U_r$, $U_{r,s}$, $\~U_{r,s}$ are independent of the choice of $X$. 
\end{definition}

\begin{remark} The significance of the open subschemes of the form $U_r$ and $U_{r,s}$ may be understood by recalling the observations in Remark \ref{remark homotopies on ruled surfaces}.
\end{remark} 

\begin{lemma}
\label{lemma elementary transformation}
Let $r,s \geq 1$ be rational numbers.
\begin{enumerate}[label=$(\alph*)$]
\item The schemes $U_{r}$ and $U_{r-1}$ are isomorphic over $C$. If $r$ is an integer, then $\~U_r$ and $\~U_{r-1}$ are isomorphic over $C$. 
\item Let $r$ and $s$ be such that if they are written in reduced form as $r = a/b$ and $s = c/d$, then $r-s = 1/(bd)$. Then the schemes $U_{r,s}$ and $U_{r-1,s-1}$ are isomorphic over $C$. If $r$ is an integer, then $\~U_{r,s}$ and $\~U_{r-1,s-1}$ are isomorphic over $C$. 
\end{enumerate}
\end{lemma}

\begin{proof}
We will sketch the argument for the first part of (a); the proofs of the remaining statements are essentially the same and are left to the reader. Let $W = X_1 \backslash \ell_0$. As $r>1$, the image of $U_r \to X_1$ factors through $W$. However, observe that the line $\ell_0$ in $X_1$ has self-intersection $-1$ and can thus be contracted due to Castelnuovo's contractibility criterion. Indeed, there exists a morphism $X_1 \to \#P^1_U$ for which $\ell_0$ is the exceptional divisor. (This process of blowing up a point on a ruled surface and contracting the strict transform of the fiber through the point is called an \emph{elementary transformation}; see \cite[Chapter V, Example 5.7.1]{Hartshorne}.) Thus, we see that there is an isomorphism $\theta: W \to \#P^1_U \backslash \{(c_0, [0:1])\}$.  It is easy to see that 
\[
U_{r-1} \cong U_r \times_{W, \theta} (\#P^1_U \backslash \{(c_0, [0:1])\}).
\]
Indeed, an explicit computation shows that $\theta^*(y) = x/y$. Thus, one can see that for any positive integers $a$ and $b$, $\theta^*(x^a/y^b) = x^{a-b}/y^b$.  
\end{proof}

Lemma \ref{lemma elementary transformation} tells us that the local geometry of $U_r$, where $r$ is a non-negative integer is the same as that of $U_0$.  However, an exact analogue is not true for $U_r$, where $r$ is not an integer.   The next result shows that the pullback of $U_r \to C$ by an appropriate finite morphism (induced by the taking the power of $x$ by the denominator of $r$ in the reduced form) has the same local geometry as that of $U_0$,  after performing some universal constructions described in Section \ref{subsection constructions}.

\begin{theorem}
\label{theorem etale cover} 
Let $r$ be a non-negative rational number, written in reduced rational form as $a/b$. Let $\phi_b: C \to C$ be the morphism corresponding to the $k$-homomorphism $A \to A$ defined by $x \mapsto x^b$. Let $p_1: U_r \times_{C, \phi_b} C \to U_r$ and $p_2: U_r \times_{C, \phi_b} C \to C$ be the two projection morphisms. We view $U_r \times_{C, \phi_b} C$ as a $C$-scheme via $p_2$ and denote $p_2^*(x)$ by $x$. Then, there exists an isomorphism of $C$-schemes 
\[
U_{a} \stackrel{\sim}{\to} (U_{r} \times_{C,\phi_b} C)[x^a/y], 
\]
where $(U_{r} \times_{C,\phi_b} C)[x^a/y]$ is considered a $C$-scheme via the morphism
\[
(U_{r} \times_{C,\phi_b} C)[x^a/y] \to U_{r} \times_{C,\phi_b} C \stackrel{p_2}{\to} U \text{.}
\]
Moreover, the composition morphism 
\[
(U_r \times_{C, \phi_b} C)[x^a/y]  \to (U_r \times_{C, \phi_b} C) \to U_r
\]
is a $b$-sheeted finite \'etale cover. 
\end{theorem}

\begin{proof}
The idea of the proof is to consider a particular Zariski open cover of $U_r$ and then demonstrate the isomorphism given in the statement of the theorem over every piece of the cover using universal properties from Section \ref{subsection constructions}.  It will be convenient to embed $U_r$ into an appropriate nodal blowup $X$.  We begin by choosing a specific $X \in \@N$ containing the line $\ell_r$ for an explicit description of $U_r$.  

We first set $s_0 = 0$, $s_1 = 1$ and $s_2 = 1/2$. For $i > 2$ we choose $s_i$ as follows, with $a_i/b_i$ denoting the reduced rational representation of $s_i$ for each $i$:
\begin{itemize}
\item If $r$ lies between $s_{i-2}$ and $s_{i-1}$, we define $s_{i} := \frac{a_{i-2} + a_{i-1}}{b_{i-2} + b_{i-1}}$. 
\item If $r$ does not lie between $s_{i-1}$ and $s_i$, then we define $s_{i} := \frac{a_{i-3} + a_{i-1}}{b_{i-3} + b_{i-1}}$.  
\end{itemize}
It can be proved that this sequence $\{s_i\}_i$ is finite and terminates in $r$ (see \cite[Section 4.5]{GKP}). Suppose $s_n = a/b = r$. Also, if the numbers $s_i$ are arranged in increasing order and if $s_i$ and $s_j$ are adjacent to each other in this arrangement with $s_i < s_j$, then we have $a_j b_i - a_i b_j = 1$. 

Thus, with the notation established in Section \ref{subsection constructions}, we have $$X = \#P^1_C[x^{a_1}/y^{b_1}, \ldots, x^{a}/y^{b}].$$  Using Lemma \ref{lemma zeros and poles}, it is easy to see that 
\[
 \#P^1_C[x^{a_1}/y^{b_1}, \ldots, x^{a}/y^{b}] \{x^a/y^b\} = U_{r} \backslash (\ell_{\infty} \cup \ell_{-\infty}). 
\]
We denote this open subscheme of $X$ by $W_1$. Let $W_2$ denote the open subscheme $X \times_{C} (C \backslash \{c_0\}) = X\{x\}$. It is clear that $\{W_1, W_2\}$ is a Zariski open cover of $U_r$. 

Note that on $W_1 \times_{U, \phi_b} U$, we have $p_1^*(x) = p_2^*(\phi_b^*(x))  = p_2^*(x)^b$. Since we are writing $p_2^*(x)$ as just $x$, we may write $p_1^*(x) = x^b$. Thus, we compute
\begin{align*}
(W_1 \times_{C, \phi_b} C) & \simeq \#P^1_C[x^{a_1}/y^{b_1}, \ldots, x^{a}/y^{b}] \{x^a/y^b\}) \times_{C, \phi_b} C \\
& \simeq \#P^1_C[p_1^*(x)^{a_1}/y^{b_1}, \ldots, p_1^*(x)^a/y^b] \{p_1^*(x)^a/y^b\}\\
& \simeq \#P^1_C[x^{a_1b}/y^{b_1}, \ldots,x^{ab}/y^b] \{x^{ab}/y^b\}.
\end{align*}
Hence,
\begin{align*}
(W_1 \times_{C, \phi_b} C)[x^a/y] & \simeq  \#P^1_C[x^{a_1b}/y^{b_1}, \ldots, (x^{a}/y)^b]\{(x^{a}/y)^b\}[x^a/y] \\
& \simeq \#P^1_C[(x^a/y)]\{(x^a/y)^b\}[x^{a_1b}/y^{b_1}, \ldots, x^{a_{n-1}b}/y^{b_{n-1}}]\\
& \simeq \#P^1_C[(x^a/y)]\{(x^a/y)\}[x^{a_1b}/y^{b_1}, \ldots, x^{a_{n-1}b}/y^{b_{n-1}}].
\end{align*}
However, for $1 \leq i \leq n-1$, we have
\[
x^{a_ib}/y^{b_i} = x^{(a_ib-ab_i)}(x^a/y)^{b_i}. 
\]
As the rational functions $x^a/y$ and $x$ extend to morphisms from $\#P^1_C[(x^a/y)]\{x^a/y\}$ to $\#P^1_C$ and since $x^a/y$ is a unit on $\#P^1_C[(x^a/y)]\{x^a/y\}$, we obtain isomorphisms
\begin{equation}
\label{eqnW1}
W_1 \times_{C, \phi_b} C [x^a/y] \xrightarrow{\sim} \#P^1_C[(x^a/y)]\{x^a/y\} \xrightarrow{\sim} U_a \backslash \left(\ell_{-\infty} \cup \ell_{\infty}\right).
\end{equation}

We have a similar computation for $(W_2 \times_{C, \phi_b} C)[x^a/y]$. 
\begin{align*}
(W_2 \times_{C, \phi_b} C)[x^a/y] & \simeq \#P^1_C[x^{a_1}/y^{b_1}, \ldots, x^{a}/y^{b}] \{x\}) \times_{U, \phi_b} U \\
& \simeq \#P^1_C[p_1^*(x)^{a_1}/y^{b_1}, \ldots, p_1^*(x)^a/y^b] \{p_1^*(x)\}\\
& \simeq \#P^1_C[x^{a_1b}/y^{b_1}, \ldots,x^{ab}/y^b] \{x^{b}\} \\
& \simeq \#P^1_C\{x^b\}[x^{a_1b}/y^{b_1}, \ldots,x^{ab}/y^b] \\
& \simeq \#P^1_C\{x\}[1/y^{b_1}, \ldots, 1/y^b] \\
& \simeq \#P^1_C \{x\}. 
\end{align*}
A similar computation shows that $U_a\{x\} \simeq \#P^1_{C \backslash \{c_0\}}$. Thus, we have
\begin{equation}
\label{eqnW2}
(W_2 \times_{C, \phi_b} C)[x^a/y] \simeq U_a \{x\}. 
\end{equation}
The open subschemes $U_a \backslash \left(\ell_{-\infty} \cup \ell_{\infty}\right)$ and $U_a \{x\}$ form a Zariski open cover of $U_a$. Thus, we have proved the required isomorphism over the pieces of a Zariski open cover. It remains to be shown that these isomorphisms agree over the intersection. The intersection of these two open subschemes of $U_a$ is 
\[
U_a \{x\} \backslash \left(\ell_{-\infty} \cup \ell_{\infty}\right) = U_a\{x\}\{y\}. 
\]
Let $\phi_1$ denote the composition
\[
U_a \{x\}\{y\} \hookrightarrow U_a\{x\} \xrightarrow{\sim} (W_2 \times_{C, \phi_b} C)[x^a/y] \hookrightarrow (U_r \times_{C, \phi_b} C)[x^a/y]
\] 
and let $\phi_2$ denote the composition
\[
U_a \{x\}\{y\} \hookrightarrow U_a\{x^a/y\} \xrightarrow{\sim} (W_1 \times_{C, \phi_b} U_r)[x^a/y] \hookrightarrow (U_r \times_{C, \phi_b} C)[x^a/y].
\]
These are open immersions with the same image.  For both $i = 1$ and $2$, the morphisms 
\[
U_a \{x\}\{y\} \stackrel{\phi_i}{\to} U_r \times_{U, \phi_b} U \to U_r
\]
express $U_a \{x\}\{y\}$ as the universal solution of the problem of
\begin{itemize}
\item attaching a $b$-th root of $x$ to $U_r$, the distinguished root being $x$ (that is, the pullback of $x$ via the obvious morphism $U_a \to U$),
\item turning $x$ into a unit, and
\item turning $y$ into a unit.
\end{itemize}
It follows from the universality of the solution that the isomorphisms \eqref{eqnW1} and \eqref{eqnW2} glue to give an isomorphism $U_{a} \stackrel{\sim}{\to} (U_{r} \times_{U,\phi_b} U)[x^a/y])$. 

It remains to prove that the projection $U_r \times_{C, \phi_b} C \to U_r$ is \'etale. It is clearly \'etale away from the zero divisor of $x$ and thus it is enough to prove that the map $(W_1 \times_{C, \phi_b} C)[x^a/y] \to W_1$ is \'etale. 

First observe that $W_1 \times_{C, \phi_b} C \simeq W_1[x^{1/b}]$ (however, note that under this isomorphism, the rational function $x^{1/b}$ on the scheme on the right corresponds to the rational function $x$ on the scheme on the left!).  Hence, we have an isomorphism $(W_1 \times_{C, \phi_b} C)[x^a/y] \simeq W_1[x^{1/b}, (x^{1/b})^a/y]$.  So, we will now show that the morphism
$W_1[x^{1/b}, (x^{1/b})^a/y] \to W_1$ is \'etale. 

Our construction of $X$ at the beginning of the proof shows that there exists a pair of integers $(c,d)$ such that $ad-bc = 1$ and such that $x^c/y^d$ is a morphism on $X$, and hence on $W_1$. Thus,
\[
(y^d/x^c)^b = (y^{bd}/x^{ad}) \cdot x. 
\]
We will now show that
\[
W_1[x^{1/b}, (x^{1/b})^a/y] \simeq W_1[(y^{bd}/x^{ad})^{1/b}]. 
\] 
As $(y^{bd}/x^{ad}) = (y^b/x^a)^d$ is a unit on $W_1$, the map $W_1[((y^b/x^a)^d)^{1/b}] \to W_1$ is \'etale and this will complete the proof of the theorem. 

On the scheme $W_1[x^{1/b}, (x^{1/b})^a/y]$, the rational function $y/(x^{1/b})^a$ induces a morphism from $W_1[x^{1/b}, (x^{1/b})^a/y]$ to $\#P^1_C$ and  
\[
((y/(x^{1/b})^a)^d)^b = (y^{bd}/x^{ad}).
\]
Thus, there is a unique morphism $\phi: W_1[x^{1/b}, (x^{1/b})^a/y] \to W_1[((y^{bd}/x^{ad}))^{1/b}]$ such that $\phi^*((y^{bd}/x^{ad})^{1/b}) = y^{d}/(x^{1/b})^{ad}$.  On the scheme $W_1[(y^{bd}/x^{ad})^{1/b}]$, we have a rational function denoted by $(y^{bd}/x^{ad})^{1/b}$ which induces a morphism from $W_1[(y^{bd}/x^{ad})^{1/b}]$ to $\#P^1_C$, and which has the property that $$((y^{bd}/x^{ad})^{1/b})^b = y^{bd}/x^{ad} = (y^b/x^a)^d.$$ As the rational function $y^b/x^a$ is a unit, so is $(y^{bd}/x^{ad})^{1/b}$. We will denote its multiplicative inverse by $(x^{ad}/y^{bd})^{1/b}$. Observe that
\[
\left[(x^{ad}/y^{bd})^{1/b} \cdot (y^d/x^c)\right]^b = (x^{ad}/y^{bd}) \cdot (y^{bd}/x^{bc}) = x. 
\]
Thus, $(x^{ad}/y^{bd})^{1/b} \cdot (y^d/x^c)$ is a $b$-th root of $x$. Thus, we obtain a $W_1$-morphism 
\[
\~{\psi}: W_1[(y^{bd}/x^{ad})^{1/b}] \to W_1[x^{1/b}]
\]
such that $\~{\psi}^*(x^{1/b}) = (x^{ad}/y^{bd})^{1/b} \cdot (y^d/x^c)$. Observe that
\[
\left[(x^{ad}/y^{bd})^{1/b} \cdot (y^d/x^c) \right]^a \cdot (1/y) = \left[(x^{ad}/y^{bd})^{1/b}\right]^a \cdot (y^b/x^a)^c.
\]
Since $(x^{ad}/y^{bd})^{1/b}$ defines a morphism from $W_1[(y^{bd}/x^{ad})^{1/b}]$ to $\#P^1_C$ and since $(y^b/x^a)$ is a unit, it follows that the rational function $(\~{\psi}^*(x^{1/b}))^a/y$ defines a $W_1[x^{1/b}]$-morphism from $W_1[(y^{bd}/x^{ad})^{1/b}]$ to $\#P^1_C$. Thus, $\~{\psi}$ lifts uniquely to a morphism 
\[
\psi: W_1[(y^{bd}/x^{ad})^{1/b}] \to W_1[x^{1/b}, (x^{1/b})^a/y].
\] 
Observe that
\begin{align*}
\psi^*(\phi^*((y^{bd}/x^{ad})^{1/b})) & = \psi^*(y^d/(x^{1/b})^{ad}) \\
& = \frac{y^d}{\left[(x^{ad}/y^{bd})^{1/b} \cdot (y^d/x^c)\right]^{ad}}\\
& = y^d  \left(\frac{x^{acd}}{y^{ad^2}}\right) \left[\left(\frac{y^{bd}}{x^{ad}} \right)^{1/b}\right]^{ad}\\
& = \left(\frac{x^{acd}}{y^{d(ad-1)}}\right)  \left[\left(\frac{y^{bd}}{x^{ad}} \right)^{1/b}\right]^{(ad-1)} \left(\frac{y^{bd}}{x^{ad}} \right)^{1/b}\\
& = \left(\frac{y^{bd}}{x^{ad}} \right)^{1/b},
\end{align*}
since $ad - bc =1$.  Thus, by the universal property of the scheme $W_1[(y^{bd}/x^{ad})^{1/b}]$, we see that $\phi \circ \psi$ is the identity morphism of $W_1[(y^{bd}/x^{ad})^{1/b}]$.  A similar computation yields
\[
\phi^*(\psi^*(x^{1/b}))  = \phi^*((x^{ad}/y^{bd})^{1/b} \cdot (y^d/x^c))  = \left(\frac{(x^{1/b})^{ad}}{y^d}\right) \left(\frac{y^d}{x^c} \right) = x^{1/b}. 
\]
Thus, by the universal property of the scheme $W_1[x^{1/b}, (x^{1/b})^a/y]$, we see that $\psi \circ \phi$ is the identity morphism of $W_1[x^{1/b}, (x^{1/b})^a/y]$. Thus, $\phi$ and $\psi$ are isomorphisms, which concludes the proof of the theorem. 
\end{proof}

\section{The special case of nodal blowups}
\label{section nodal blowups}

We will keep the conventions and notation from Section \ref{subsection nodal blowups of ruled surfaces}.  The aim of this section is to show that Theorem \ref{theorem ruled iterations of S} holds for nodal blowups, that is, elements of $\@N$.

\begin{theorem}
\label{theorem nodal case}
Let $X \in \@N$ and let $U$ be a henselian local scheme over $k$. Then $$\pi_0^{\#A^1}(X)(U) = \mathcal S^2(X)(U) = \mathcal S^n(X)(U),$$ for all $n \geq 2$. 
\end{theorem}

Theorem \ref{theorem nodal case} is a direct consequence of Proposition \ref{proposition restricted nodal} and Proposition \ref{proposition nodal reduction} below.  It will be convenient to first prove Theorem \ref{theorem nodal case} in the case when $X$ belongs to a special subset of $\@N$, which we describe in Notation \ref{notation N'}. 

Let $X \in \@N$. This means that there is a sequence of morphisms 
\begin{equation}
\label{eqn blowup}
X = X_p \to \cdots \to X_1 \to X_0 = \#P^1_C, 
\end{equation}
where for all $i \geq 1$, $X_{i} \to X_{i-1}$ is a blowup at some node. The discussion in Section \ref{subsection nodal blowups of ruled surfaces} shows that there exist ordered pairs of non-negative integers $(m_0, n_0), (m_1,n_1), \ldots, (m_p, n_p)$, $(m_{p+1}, n_{p+1})$ with the following properties:
\begin{itemize}
\item[(a)] $m_i$ and $n_i$ are coprime for each $i$; 
\item[(b)] $(m_0, n_0) = (0,1)$, $(m_1, n_1) = (1,0)$;
\item[(c)] For any $i$, there exist integers $j$ and $k$ such that $1 \leq j < k< i$ such that $m_k n_j - m_j n_k = 1$ and $m_i = m_j + m_k$, $n_i = n_j + n_k$;
\item[(d)] $X_i = X_{i-1}[x^{m_{i+1}}/y^{n_{i+1}}]$ (in the sense of Proposition \ref{proposition resolving indeterminacies}).
\end{itemize}
Note that $n_i \neq 0$ for any $i \neq 1$.  From this description, it is clear that $X$ can be obtained from $\#P^1_C$ by blowing up the ideal $\@I = \prod_{i=1}^p \<x^{m_i}, y^{n_i}\>$.  The lines on $X$ in the sense of Notation \ref{notation nodal blowups} are the $\ell_{m_i/n_i}$.  

We will need to keep track of pairs of integers $(i,j)$ such that the pseudo-lines $\ell_{m_i/n_i}$ and $\ell_{m_j/n_j}$ intersect (this will be important in Proposition \ref{proposition nodal reduction}).  Thus, it will be convenient to relabel these ordered pairs as $(a_0, b_0)$, $(a_1, b_1), \ldots$, $(a_{p+1},b_{p+1})$, in such a way that $a_0/b_0 = 0 < a_1/b_1 < \ldots <a_p/b_p$ and $(a_{p+1}, b_{p+1}) = (1,0)$.  Set $s_i = a_i/b_i$ for $0 \leq i \leq p$ and $s_{p+1} = \infty$.   Note that $s_p = a_p/b_p$ is always an integer.   There are exactly $p+2$ pseudo-lines on $X$, namely $\ell_{0} = \ell_{s_0}$, $\ell_{s_1}$, $\cdots$, $\ell_{s_p}$ and $\ell_{s_{p+1}} = \ell_{\infty}$. For $i < j$, the pseudo-lines $\ell_{s_i}$ and $\ell_{s_{j}}$ meet if and only if $j = i+1$. The node in which $\ell_{s_i}$ and $\ell_{s_{i+1}}$ meet is locally given by the ideal $\<x^{a_{i+1}}/y^{b_{i+1}}, y^{b_{i}}/x^{a_{i}}\>$.  %From this description, it is clear that $X$ can be obtained from $\#P^1_C$ by blowing up the ideal $\@I = \prod_{i=1}^p \<x^{a_i}, y^{b_i}\>$. 

\begin{notation}
\label{notation N'}
We let $\@N'$ denote the subset of $\@N$ consisting of blowups $X$ of $\P^1_C$ such that there exists a sequence of ordered pairs of non-negative integers $(a_0,b_0), \ldots, (a_{p+1}, b_{p+1})$ such that:
\begin{itemize}
\item[(1)] $(a_0,b_0) = (0,1)$ and $(a_{p+1}, b_{p+1}) = (1,0)$, 
\item[(2)] $a_i/b_i \leq 1$ for all $1 \leq i \leq p$,
\item[(3)] if $s_i = a_i/b_i$ for $1 \leq i \leq p$, then $\ell_0, \ell_{s_1}, \ldots, \ell_{s_p}, \ell_{\infty}$ are all the pseudo-lines on $X$.  
\end{itemize}  
\end{notation}

We first prove Theorem \ref{theorem nodal case} by restricting to this special subset $\@N'$ of nodal blowups in Proposition \ref{proposition restricted nodal} below and then reduce the general case to this one by means of elementary transformations (in the form of Lemma \ref{lemma elementary transformation}) in Proposition \ref{proposition nodal reduction}.  

\begin{proposition}
\label{proposition restricted nodal}
Let $X \in \@N'$ and let $U$ be a henselian local scheme over $k$. Then $$\pi_0^{\#A^1}(X)(U) = \mathcal S^2(X)(U) = \mathcal S^n(X)(U),$$ for all $n \geq 2$.
\end{proposition}

\begin{proof}
We use the notation that has been set up above. Thus, there exists a sequence of ordered pairs of non-negative integers $(a_0,b_0), \ldots, (a_{p+1}, b_{p+1})$ and morphisms
\[
X = X_p \to \cdots \to X_1 \to X_0 = \#P^1_C
\]
such that $X_i = X_{i-1}[x^{a_i}/y^{b_i}]$ and $a_i/b_i \leq 1$ for all $1 \leq i \leq p$ with $(a_p, b_p)=(1,1)$.  By Remark \ref{remark trivial case}, we are reduced to determining when two $\gamma$-morphisms $U \to \P^1_C$ that lift to $X$ and map the closed point of $U$ to $(c_0, [0:1])$ are $n$-ghost homotopic for $n \geq 1$. 

Write $U = \Spec R$, where $R$ is a smooth henselian local ring with maximal ideal $\mathfrak{m}$ and write $r_0 = \gamma^*(x) \in R$ for the pullback of the uniformizing parameter $x$ on the coordinate ring of $C$. We work on $X_{\gamma}$ defined by the pullback
\[
\xymatrix{
X_{\gamma} \ar[r] \ar[d] & X \ar[d] \\
U \ar[r]^-{\gamma}        & C
}
\]
and show that any two sections $U \to X_{\gamma}$ which are $n$-ghost homotopic for some $n$ are actually $1$-ghost homotopic and also map to the same section of $\pi_0^{\#A^1}(X_{\gamma})(U)$. Any $\gamma$-morphism $\alpha: U \to \#P^1_C$ induces a section $\beta: U \to \#P^1_U$. If the $\gamma$-morphism $\alpha$ maps the closed point $u$ of $U$ to $(c_0, [0:1])$, the corresponding section $\beta: U \to \#P^1_U$ maps $u$ to $(u, [0:1])$. Hence it factors through the open subscheme $\#A^1_U = \Spec R[y] = \Spec R[Y/Z] \subset \#P^1_U$. Thus, any such $\beta$ is characterized by an $R$-algebra homomorphism $\beta^*: R[y] \to R$. Clearly, $\beta^*(y) = \alpha^*(y)$. Moreover, $\alpha$ maps $u$ to $(c_0, [0:1])$ if and only if $\alpha^*(y)$ lies in the maximal ideal $\mathfrak{m}$ of $R$. 

For any $r \in R$, let $\beta_r: U \to \#P^1_U$ be the morphism induced by the $R$-algebra homomorphism $R[y] \to R$, $y \mapsto r$. Note that any section $\beta: U \to \#P^1_U$ such that $\beta(u) = (u, [0:1])$ is of the form $\beta_r$ for some unique $r \in \mathfrak{m}$. Such a section lifts to $X_{\gamma}$ if and only if the ideal $\<r_0^{a_i}, r^{b_i}\>$ is principal for every $i$, $1 \leq i \leq p$. 

Let $r_1, r_2 \in \mathfrak{m}$ be such that the sections $\beta_{r_1}$ and $\beta_{r_2}$ lift to $X_{\gamma}$ and are connected by an $n$-ghost homotopy $\@H$ which also lifts to $X_{\gamma}$. As $(a_p,b_p) = (1,1)$, the sections $\beta_{r_1}$, $\beta_{r_2}$  and the $n$-ghost homotopy $\@H$ lift to $X_1 \times_{U, \gamma} \#P^1_U$ which is obtained from $\#P^1_U$ by blowing up the ideal $I_1 =\<r_0, y\>$. By Proposition \ref{proposition local generator ghost homotopy}, we see that the the ideals (or ideal sheaves) $\beta_{r_1}^*(I_1)$, $\beta_{r_2}^*(I_1)$, $h_{\@H}^*(I_1)$ are all generated by the pullbacks of a fixed element of $R$, which we denote by $r$.  Thus, we have the equality of ideals $\<r_0, r_1\> = \<r_0, r_2 \>$ of $R$ and moreover, this ideal is principal.  Since $R$ is a local domain, $\<r_0, r_i\>$ is principal if and only if $r_0 | r_i$ or $r_i|r_0$, for $i=1,2$.  There are clearly two possibilities:
\begin{itemize}
\item[(A)] $r_0 | r$: This happens when the ideals $\alpha_1^*(I_1)$, $\alpha_2^*(I_1)$ are generated by $r_0$. Thus $r_0|r_1$ and $r_0|r_2$. 
\item[(B)] $r|r_0$ but $r_0 \notdivides r$: In this case, for $i = 1,2$, the ideal $\alpha_i^*(I_1) = \<r_0, r_i\>$ is generated by $r_i$. As $\<r_0, r_1\> = \<r_0, r_2\>$, we see that $r_1$ and $r_2$ must be unit multiples of $r$.   
\end{itemize} 
At least one of (A) or (B) must hold for the lifts of $\beta_{r_1}$ and $\beta_{r_2}$ to $X_{\gamma}$ to be $n$-ghost homotopic. Of course, this is only a necessary condition for the lifts to be $n$-ghost homotopic. We will see that the condition in case (A) is actually sufficient for the two lifts to be $\#A^1$-chain homotopic. However, in case (B), an additional condition is required. 

Case (A) is very easy to deal with. Consider the homotopy 
\[
h: \Spec R[T] = \#A^1_U \to \#A^1_U = \Spec R[y] 
\]   
defined by $y \mapsto r_1T + r_2(1-T)$. Then as $a_i \leq b_i$ for all $i$, and as $r_0$ divides $r_1$ and $r_2$, we get
\begin{align*}
h^*(\<r_0^{a_i}, y^{b_i}\>) & = \<r_0^{a_i}, (r_1T + r_2(1-T))^{b_i}\> \\
& = \<r_0^{a_i} \> 
\end{align*}
which is a principal ideal. Thus, the homotopy $h$ lifts to $X_{\gamma}$. Thus, we see that condition (A) is actually sufficient for the lifts of $\beta_{r_1}$ and $\beta_{r_2}$ to be $\#A^1$-chain homotopic. Thus, in this case they also map to the same element of $\pi_0^{\#A^1}(X_{\gamma})(U)$. 

Now, we consider case (B). Thus, $r$ is an element of $R$ such that $r|r_0$, $r_0 \notdivides r$ and $r_1$ and $r_2$ are unit elements of $\mathfrak{m}$ such that $r_i/r$ is a unit for $i = 1,2$. We assume that $\beta_{r_i}$ lifts to $X_{\gamma}$ for $i = 1,2$.  Claim $1$ below gives a necessary algebraic condition for $\beta_{r_1}$ and $\beta_{r_2}$ to be $n$-ghost homotopic.  Claim $2$ below will show that the algebraic condition given in Claim $1$ is sufficient for $\beta_{r_1}$ and $\beta_{r_2}$ to be $1$-ghost homotopic.  This is the trickiest part of the proof and involves constructing an explicit $1$-ghost homotopy over an elementary Nisnevich cover of $\A^1_U$ (in fact, a Zariski cover) so that we can apply Lemma \ref{lemma S^2 agrees with pi0} to conclude that the above sections map to the same element of $\pi_0^{\#A^1}(X_{\gamma})(U)$.

\medskip

\noindent \emph{\bf Claim 1:}  If there exists an $n$-ghost homotopy $\mathcal H$ connecting $\beta_{r_1}$ to $\beta_{r_2}$ which lifts to $X_{\gamma}$, then
\[
\frac{r_2}{r_1}-1 \in \sqrt{\<r, r_0/ r\>}.
\]
\noindent \emph{Proof of Claim 1:} Suppose there exists an $n$-ghost homotopy $\mathcal H$ connecting $\beta_{r_1}$ to $\beta_{r_2}$ which lifts to $X_{\gamma}$. By Proposition \ref{proposition local generator ghost homotopy}, the ideal sheaf $h^{*}_{\mathcal H}(I_1)$ is generated by $f_{\mathcal H}^*(r)$. At any point $v$ of $\Sp(\mathcal H)$, the ideal $h_{\mathcal H}^*(I_1)_v$ is equal to $\<f_{\mathcal H}^*(r_0), h^{*}_{\mathcal H}(y)\>$. Thus, we see that the ideal $I(r):=\<f_{\mathcal H}^*(r), h^{*}_{\mathcal H}(y)\>$ of $\mathcal O_{\Sp(\mathcal H),v}$ is principal. Thus if $X_{\gamma,r}$ denotes the scheme obtained by blowing up $\mathbb P^1_U$ at the closed subscheme $\@Z(h_{\@H}^*(\@I_{\gamma}) \cdot I(r))$ we see that the $n$-ghost homotopy $\@H$ lifts uniquely to $X_{\gamma,r}$. Let us denote this lift by $\@H^r$. The scheme $X_{\gamma,r}$ can also be constructed by first blowing up the closed subscheme $\mathcal Z(I(r))$ to construct $X(r)$ and then blowing up the total transform of the ideal sheaf $h_{\mathcal H}^*(\@I_{\gamma})$ on $X(r)$. We will now examine this total transform. 

We view $X(r)$ as a closed subscheme of $U \times \mathbb P^1 \times \mathbb P^1$ where we use the homogeneous coordinates $Y,Z$ for the first copy of $\mathbb P^1$ and $Y_0, Y_1$ for the second copy. Then, $X(r)$ is given by the equation $rY_0Z = Y_1Y$. It suffices to compute the total transform of $h_{\mathcal H}^*(\@I_{\gamma})$ in the open patch $Z \neq 0$ (since both $I(r)$ and $\@I_{\gamma}$ have support within this patch). 

We will now show that there exist closed subschemes $T_1$ and $T_2$ of $X(r)$ such that:
\begin{itemize}
\item[(a)] The blowup of $X(r)$ at $T_1 \cup T_2$ is isomorphic to $X_{\gamma,r}$ over $X(r)$. 
\item[(b)] $T_1$ is supported in the closed subscheme $S_1$ of $X(r)$ which is defined in the open patch $Y_1 \neq 0$ by the ideal $\<r_0/r, Y_0/Y_1\>$. 
\item[(c)] $T_2$ is supported in the closed subscheme $S_2$  of $X(r)$ which is  defined in the open patch $Y_0 \neq 0$ by the ideal $\<r, Y_1/Y_0\>$. 
\end{itemize}

For this it will suffice to show that for each $i$, the total transform $I_i$ is the product of a locally principal ideal with an ideal $J_i$ where $J_i$ is supported in either $S_1$ or $S_2$. We will consider two cases - either $r_0^{a_i}|r^{b_i}$ or $r^{b_i}|r_0^{a_i}$ (which are not mutually exclusive). 

Suppose $r_0^{a_i}|r^{b_i}$. Notice that since we already know that $r|r_0$, this implies that $r_0$ and $r$ have the same squarefree part. On the open patch where $Y_1 \neq 0$,   
\begin{align*}
\<r_0^{a_i}, y^{b_i}\>  & =  \<r_0^{a_i}\>\;\< 1, (r^{b_i}/r_0^{a_i})(Y_0/Y_1)^{b_i}\> \\
                  & =  \< r_0^{a_i}\>
\end{align*}
which is principal. On the open patch where $Y_0 \neq 0$, we simply observe that since $r$ and $r_0$ have the same squarefree part, the ideal $\<r_0^{a_i}, (Y_1/Y_0)^{b_i}\>$ has the same radical as $\<r, Y_1/Y_0\>$ and the associated closed scheme is contained in $T_2$. 

If $r^{b_i}|r_0^{a_i}$, then on the open patch $Y_1 \neq 0$, we have
\[
\<r_0^{a_i}, y^{b_i}\> = \<r^{b_i}\> \;\<(r_0^{a_i}/r^{b_i}), (Y_0/Y_1)^{b_i}\> \text{.}
\]
We claim that the support of the ideal $\<(r_0^{a_i}/r^{b_i}), (Y_0/Y_1)^{b_i}\>$ is contained in $T_1$. For this, it suffices to show that any prime $\mathfrak p$ of $R$ which does not contain $r_0/r$ cannot contain $r_0^{a_i}/r^{b_i}$. To see this, observe for any such prime $\mathfrak p$, $r_0/r$ is a unit in $R_{\mathfrak p}$ then
\[
(r_0)^{a_i}/r^{b_i} = (r_0/r)^{a_i} (1/r)^{b_i-a_i} \text{.}
\]
Since $(r_0)^{a_i}/r^{b_i} \in R_{\mathfrak p}$, this shows that $(1/r)^{b_i - a_i} \in R_{\mathfrak p}$. Since $r \in R_{\mathfrak p}$, it follows that $r$ is actually a unit in $R_{\mathfrak p}$. Hence $r_0^{a_i}/r^{b_i}$ is also a unit in $R_{\mathfrak p}$, as required. 

Observe that the projection of $X(r) \subset U \times \#P^1 \times \#P^1$ onto the first and third factors can be viewed as the blowup of $\#P^1_U = \Proj R[Y_0, Y_1]$ at the closed subscheme $\<r, Y_1\>$. Thus, the projection map $\theta: X_{\gamma, r} \to \Proj R[Y_0,Y_1]$ is just the blowup of $\Proj R[Y_0,Y_1]$ at the closed subschemes $\mathcal Z(\<r,Y_1\>)$ and $\mathcal Z(\<r_0/r,Y_0\>)$. Let $\mathcal H^{\prime}$ be the $n$-ghost homotopy $\theta \circ \mathcal H^r$. Then if $\beta^{\gamma,r}_{r_i}$ denotes the lift of $\beta_{r_i}$ to $X_{\gamma,r}$, we see that $\theta \circ \beta^{\gamma,r}_{r_1}$ is the morphism corresponding to the $R$-algebra homomorphism $R[Y_0/Y_1] \to R$, $Y_0/Y_1 \mapsto r_1/r$. Thus, by Proposition \ref{proposition local generator ghost homotopy}, we see that $\mathcal H^{\prime}$ avoids the closed subschemes $\mathcal Z(\<r,Y_1\>)$ and $\mathcal Z(\<r_0/r, Y_0\>)$.  Let $U(r,r_0/r)$ denote the closed subscheme of $U$ corresponding to the ideal $\<r, r_0/r\>$ fo $R$.  Thus, the restriction of the $n$-ghost homotopy $\mathcal H^{\prime}$ to $U(r,r_0/r)$ factors through 
\[
\mathbb P^1_{U(r,r_0/r)} \setminus  (\{(0:1), (1:0)\} \times U(r,r_0/r))  \simeq \mathbb G_m \times U(r,r_0/r).
\]
As $\mathbb G_m$ is $\mathbb A^1$-rigid, the induced homotopy of the underlying reduced subscheme of $U$ is constant. In particular, since $\theta \circ \beta^{\gamma,r}_{r_i}: U \to \Proj R[Y_0,Y_1]$ corresponds to the homomorphism $Y_0/Y_1 \mapsto r_i/r$ for $i = 1,2$, we see that $r_2/r_1 = 1$ modulo any prime of $R$ containing the ideal $\<r,r_0/r\>$. Thus $r_2/r_1 - 1 \in \sqrt{\<r, r_0/ r\>}$. 

To summarize, we have proved that if   $r_1, r_2 \in \mathfrak{m}$ are unit multiples of $r$ such that $\beta_{r_1}$ and $\beta_{r_2}$ are $n$-ghost homotopic in $X_{\gamma}$, then $r_2/r_1-1$ lies in the ideal $\sqrt{\<r, r_0/ r\>}$. 

\noindent \emph{\bf Claim 2:} If $r_1, r_2 \in \mathfrak{m}$ are unit multiples of $r$ such that
\[
\delta:=\frac{r_2}{r_1}-1 \in \sqrt{\<r, r_0/ r\>},
\]
then $\beta_{r_1}$ and $\beta_{r_2}$ are $1$-ghost homotopic in $X_{\gamma}$.

\noindent \emph{Proof of Claim 2:}  We will explicitly construct a $1$-ghost homotopy $$\mathcal H^{\prime} := (V \to \mathbb A^1_U, W \to V \times_{\mathbb A^1_U} V, h, \~\sigma_0, \~\sigma_1, \mathcal H^W) $$ 
of $U$ in $\Proj R[Y_0,Y_1]$ connecting $\theta \circ \beta^{\gamma,r}_{r_1}$ to $\theta \circ \beta^{\gamma,r}_{r_2}$ which avoids the closed subschemes $\mathcal Z(\<r, Y_1\>)$ and $\mathcal Z(\<r_0/r, Y_0\>)$.

We first construct a Zariski open cover $V$ of $\mathbb A^1_U = \Spec R[S]$.  Let $U(r)$ denote the closed subscheme of $U$ corresponding to the ideal $\<r\>$ of $R$. Define 
\begin{eqnarray*}
V_1 & := & \mathbb A^1_U \backslash \mathcal Z(\<r_0/r, 1 + \delta S\>); \\ 
V_2 & := & \mathbb A^1_{U \backslash U(r)}.%; \\
%V_3 & := & \mathbb A^1_{U \backslash U(r^{\prime})}.
\end{eqnarray*}
\noindent This is a Zariski open cover of $\mathbb A^1_U$. Indeed, to show this, it is enough to see that if $\mathfrak{p}$ is a prime ideal of $R[S]$ containing $(r_0/r, 1 + \delta S)$, then $\delta$ is a unit modulo $\mathfrak{p}$.  Since $\delta \in \sqrt{\<r, r_0/ r\>}$, we see that $\mathfrak{p}$ cannot be a prime containing $\<r,r_0/r\>$. But then $\mathfrak{p}$ must fail to contain either $r$ or $r_0/r$.  Since $\mathfrak{p}$ contains $r_0/r$, it cannot contain $r$.  Thus the point of $\mathbb A^1_U$ corresponding to $\mathfrak{p}$ must lie in $V_2$. We define $V = V_1 \amalg V_2$. We will now define morphisms $h_i: V_i \to \Proj R[Y_0,Y_1]$ for $i = 1,2$ and obtain $h: V \to \Proj R[Y_0, Y_1]$ by defining $h_{V_i} = h_i$. 

The morphism $h_1: V_1 \to \Proj R[Y_0,Y_1]$ is given by $Y_0/Y_1 \mapsto (1+\delta S)(r_1/r)$. Notice that the assignment $Y_0/Y_1 \mapsto 1 + \delta S$ actually defines a morphism $\mathbb A^1_U \to \Proj R[Y_0,Y_1]$ which factors through the complement of $\mathcal Z(\<r,Y_1\>)$. However, it fails to avoid $\mathcal Z(\<r_0/r, Y_0\>)$. Indeed the preimage of this subscheme is the closed subscheme of $\mathcal Z(\<r_0/r, 1 + \delta S\>)$ of $\mathbb A^1_U$. Hence, we have cut out the scheme $\mathcal Z(\<r_0/r, 1 + \delta S\>)$ to ensure that the morphism $h_1: V_1 \to \Proj R[Y_0,Y_1]$ gives rise to a morphism $V_1 \to X_{\gamma,r}$. 

The morphism $h_2: V_2 \to \Proj R[Y_0,Y_1]$ is given by composing the projection $V_2 \to U \backslash U(r)$ with the morphism $U\backslash U(r) \to \mathbb P^1_U$ given by $Y_0/Y_1 \mapsto 1$, that is, it is the ``constant section at $1$". In other words, it factors through $Y_1 \neq 0$ and is given by $Y_0/Y_1 \mapsto 1 \in R[S,r^{-1}][T]$. 

The morphisms $\sigma_0, \sigma_1: U \to \mathbb A^1_U = \Spec R[S]$ factor through the open immersion $V_1 \hookrightarrow \mathbb A^1_U$. We choose the induced morphisms $U \to V_1 \hookrightarrow V$ as the morphisms $\~\sigma_i$. 
We choose $W$ to be equal to $V \times_{\mathbb A^1_U}V = \coprod_{(i,j)} (V_{ij})$ where $V_{ij} = V_i \cap V_j$. The $0$-ghost chain homotopy $\mathcal H^W$ will be a single $\mathbb A^1$-homotopy which we will define separately on each piece $V_{ij}$. For $i = j$, we simply define it to be the constant homotopy on $V_i \cap V_i = V_i$. 

We now define a homotopy between $h_1|_{V_1 \cap V_2}$ and $h_2|_{V_1 \cap V_2}$.  This morphism will be designed to factor through $Y_0 \neq 0$ and thus will avoid $\mathcal Z(\<r_0/r,Y_0\>)$. On the other hand, by the definition of $V_2$ it is clear that it will also avoid $\mathcal Z(\<r,Y_1\>) \subset \mathcal Z(\<r\>) \subset \Proj R[Y_0,Y_1]$. We define this morphism by $Y_1/Y_0 \mapsto (1+\delta S)^{-1}(1-T) + T \in R[S,{r}^{-1}(1+\delta S)^{-1}][T]$.  The restriction of $\mathcal H^W$ to $V_{12}$ is taken to be this homotopy. The restriction of $\mathcal H^W$ to $V_{21}$ is defined to be the inverse of this homotopy.

Thus we have successfully constructed an $1$-ghost homotopy in $\Proj R[Y_0,Y_1]$ which avoids the schemes $\@Z(\<r_0/r,Y_0\>)$ and $\@Z(\<r,Y_1\>)$. Thus, it avoids the closed subscheme $T_1 \cup T_2$, and hence lifts to $X_{\gamma,r}$. It is easy to see that it connects $\theta \circ \beta^{\gamma,r}_{r_1}$ and $\theta \circ \beta^{\gamma,r}_{r_2}$. The lift of this $1$-ghost homotopy to $X_{\gamma,r}$ will connect $\beta^{\gamma,r}_{r_1}$ and $\beta^{\gamma,r}_{r_2}$. Thus, on composing with the blowup morphism $X_{\gamma,r} \to X_{\gamma}$, we get the desired $1$-ghost homotopy in $X_{\gamma}$. This shows that $\beta_{r_1}$ and $\beta_{r_2}$ are $1$-ghost homotopic as claimed. 

Finally, we observe that $\beta_{r_1}$ and $\beta_{r_2}$ have the same image in $\pi_0^{\#A^1}(X_{\gamma})$ by Lemma \ref{lemma S^2 agrees with pi0}.  This observation along with Claim 1 and Claim 2 completes the proof of the proposition. 
\end{proof}

\begin{remark}
\label{remark corner lines}
The explicit $1$-ghost homotopy constructed in the above proof has a useful property which we will now record for use in the proof of Theorem \ref{theorem ruled iterations of S}. However, for the sake of ready reference we summarize the notation and the assumptions. 

Let $X$ be in $\@N'$. Let $U = \Spec(R)$ be the spectrum of a henselian local scheme with closed point $u$. We fix a morphism $\gamma: U \to C$ which maps the closed point of $U$ to $c_0$ and the generic point of $U$ to the generic point of $C$. Let $\alpha_1, \alpha_2: U \to \#P^1_C$ be morphisms admitting lifts to $X$, which we denote by  $\~{\alpha}_1$ and $\~{\alpha}_2$ respectively. We assume that $\~{\alpha}_1$ and $\~{\alpha}_2$ are connected by an $n$-ghost homotopy on $X$. Let $s_0= 0 < \cdots < s_p=1$ be rational numbers such that $\ell_{s_0}, \ldots, \ell_{s_p}$ are all the lines on $X$. Suppose that $\~{\alpha}_1$ (and hence $\~{\alpha}_2$) maps $u$ to the point of intersection of $\ell_{s_i}$ and $\ell_{s_{i+1}}$. Let $u'$ be a point of $U$ such that $\~{\alpha}_1(u') = \~{\alpha}_2(u')$. The proof of Theorem \ref{theorem nodal case} tells us that there exists a $1$-ghost homotopy $\~{\@H}$, satisfying the hypothesis of Lemma \ref{lemma S^2 agrees with pi0}, connecting $\~{\alpha}_1$ and $\~{\alpha_2}$. We claim that the morphism $h_{\~{\@H}}$ maps the entire fiber $\Sp(\~{\@H})_{u'}$ to $\~{\alpha}(u')$. 

Let $r_i = \alpha_i^*(y)$ for $i = 1,2$. Our hypothesis that $\~{\alpha}_1(u)$ is the intersection of two lines implies that we are in case (B) in the proof of Theorem \ref{theorem nodal case}. So we know that $r_1, r_2$ are unit multiples of each other. 

First we recall how the $1$-ghost homotopy $\~{\@H}$ is defined. The $1$-ghost homotopy $\@H'$ of $U$ on $\#P^1_U$ constructed in the proof of Theorem \ref{theorem nodal case} lifts to $X(r_1)$. We note that the morphism $h_{\@H'}$ actually factors through the open subscheme of $\#P^1_U$ over which the morphism $X(r_1) \to \#P^1_U$ is an isomorphism. On composing with the morphism $X(r) \to X_{\gamma} \to X$, we will get the required $1$-ghost homotopy $\~{\@H}$. Thus, it will suffice to show that the morphism $h_{\@H'}$ maps the fiber $\Sp(\@H')_{u'} = \Sp(\~{\@H})_{u'}$ to a single point. 

Since the morphism $h_{\@H'}$ is defined by putting together the morphisms $h_1$, $h_2$ and the morphism defining the homotopy $\@H^{W}$, it will suffice to show that each of these morphisms map the fiber of their domains over $u'$ to a single point. Looking at the formulas defining these morphisms, it clearly suffices to show that $\delta = (r_1/r_2) - 1$ is equal to $0$ modulo $\mathfrak{p}$ where $\mathfrak{p}$ is the prime ideal of $R$ corresponding to $u'$. 

Suppose $\alpha_1(u') \in \ell_i \backslash \left( \bigcup_{j \neq i} \ell_{s_j} \right)$ for $i \notin \{0,p\}$. Then, by Lemma \ref{lemma topological ruled surfaces}, the morphism $h_{\@H'}$ maps $\Sp(\~{\@H})$ into $\ell_i \backslash \left( \bigcup_{j \neq i} \ell_{s_j} \right)$. As $i \notin \{0,p\}$, the line $\ell_i$ meets two other lines. Thus, $\ell_i \backslash \left( \bigcup_{j \neq i} \ell_{s_j} \right)$ is $\#A^1$-rigid and so $h_{\~{\@H}}$ maps $\Sp(\~{\@H})_{u'}$ to a single point. 

Suppose $\alpha_1(u') \in \ell_1 \backslash \left(\bigcup_{j \neq p} \ell_{s_j} \right)$. Then, it follows that $\alpha_1^*(x/y)$ is a unit at $u'$. Thus, $r_0/r_1$ (and hence $r_0/r_2$) is not in $\mathfrak{p}$. Since $\alpha_1(u') = \alpha_2(u')$, we deduce that $r_0/r_1 - r_0/r_2$ is in $\mathfrak{p}$. As
\[
\frac{r_0}{r_1} - \frac{r_0}{r_2} = \frac{r_0}{r_2} \left(\frac{r_2}{r_1} - 1 \right) = (r_0/r_2)\cdot \delta \text{,}
\]
we see that $\delta$ is in $\mathfrak{p}$ as required. 

On the other hand, suppose $\alpha_1(u') \in \ell_0 \backslash \left(\bigcup_{j \neq 0} \ell_{s_j} \right)$. Then, it follows that $\alpha_1^*(y)$ is a unit at $u'$. Thus, $r_1$ (and hence $r_2$) is not in $\mathfrak{p}$. Since $\alpha_1(u') = \alpha_2(u')$, we deduce that $r_2 - r_1$ is in $\mathfrak{p}$. As
\[
r_2-r_1 = r_1 \cdot  \left(\frac{r_2}{r_1} - 1 \right) = r_2\cdot \delta \text{,}
\]
we see that $\delta$ is in $\mathfrak{p}$ as required. 
\end{remark}

\begin{proposition}
\label{proposition nodal reduction}
Let $U$ be a henselian local scheme over $k$.  If  $$\pi_0^{\#A^1}(X)(U) = \@S^2(X)(U) = \@S^n(X)(U)$$ for all $n \geq 2$ and for all $X \in \@N'$, then the same is true for all $X \in \@N$.
\end{proposition}
\begin{proof}
As in the proof of Proposition \ref{proposition restricted nodal}, we fix a morphism $\gamma: U \to C$ and restrict our attention to $\gamma$-morphisms in the sense of Definition \ref{definition homotopy respecting fibres}. As observed in Remark \ref{remark trivial case}, it suffices to focus on the case in which $\gamma$ maps the generic point of $U$ to the generic point of $C$ and the closed point $u$ of $U$ to the closed point of $C$.  Let $\alpha_1, \alpha_2:U \to X$ be $\gamma$-morphisms which are connected by an $n$-ghost homotopy which we denote by $\@H$.  

Let $\ell_0, \ell_{s_1}, \ldots, \ell_{s_p}$ denote all the lines on $X$ where $0 = s_0< s_1 < \cdots < s_p$ are rational numbers. Recall from the discussion following the statement of Theorem \ref{theorem nodal case} that $s_p$ is necessarily an integer.  Let $S := \{s \in \#Q ~|~ \alpha_1(u) \in \ell_{s}\}$. We know that $1 \leq |S| \leq 2$. By Lemma \ref{lemma topological ruled surfaces} and Remark \ref{remark homotopies on ruled surfaces}, the $n$-ghost homotopy $\@H$ connecting $\alpha_1$ and $\alpha_2$ factors through the subscheme $W = X \backslash \left(\bigcup_{s_j \notin S} \ell_{s_j} \right)$. 

\noindent
\textbf{Case A:} Suppose that $|S|=1$, say $S= \{s_i\}$. 

If $i \notin \{0,p\}$, then $\ell_{s_i}$ meets exactly two lines and we see that $W \cong U_{s_i}$ (see Definition \ref{definition U_r}). If $i \in \{0,p\}$, this complement is the open subscheme $\~U_{s_i}$ of $X$. In either case, by Lemma \ref{lemma elementary transformation}, $W$ is seen to be isomorphic to an open subscheme $W' \cong U_{s_i - \lfloor s_i \rfloor}$ of some $X' \in \@N'$. By Proposition \ref{proposition restricted nodal}, the morphisms $U \stackrel{\alpha_i}{\to} W \cong W' \hookrightarrow X'$, for $i =1,2$, are connected by a $1$-ghost homotopy which satisfies the hypothesis of Lemma \ref{lemma S^2 agrees with pi0}. Again, by Lemma \ref{lemma topological ruled surfaces} and Remark \ref{remark homotopies on ruled surfaces}, this $1$-ghost homotopy factors through $W'$. Thus, $\alpha_1$ and $\alpha_2$ are also connected by a $1$-ghost homotopy of $U$ on $X$, satisfying the hypothesis of Lemma \ref{lemma S^2 agrees with pi0}.  

\noindent
\textbf{Case B:} Now suppose that $|S| = 2$; then $S=\{s_i, s_{i+1}\}$ for some $i$.  

In this case, $W$ is either of the form $U_{s_i,s_{i+1}}$ or $\~U_{s_i, s_{i+1}}$ depending on the values of $i$. Using Lemma \ref{lemma elementary transformation}, we find that $W$ is isomorphic to $U_{t_i, t_{i+1}}$ or $\~U_{t_i, t_{i+1}}$, where $t_i = s_i-\lfloor s_i \rfloor$ and $t_{i+1} = s_{i+1}-\lfloor s_i \rfloor$. In either case, it is isomorphic to an open subscheme $W'$ of some $X' \in \@N'$. The proof of Proposition \ref{proposition restricted nodal} gives us a $1$-ghost homotopy $\@H'$ connecting $\alpha_1$ and $\alpha_2$ which satisfies the hypothesis of Lemma \ref{lemma S^2 agrees with pi0}. We claim that for any point $u'$ of $U$ such that $\gamma(u') = c_0$, $h_{\@H'}$ maps $\Sp(\@H')_{u'}$ into $W'$. 

If $\alpha_1(u') = \alpha_1(u) = \ell_{t_i} \cap \ell_{t_{i+1}}$, our claim follows immediately since Lemma \ref{lemma topological ruled surfaces} shows that $h_{\@H'}$ maps $\Sp(\@H')_{u'}$ to $\alpha_1(u)$. 

Suppose $\alpha_1(u')$ lies in $\ell_{t_{j(u')}} \backslash \left(\bigcup_{r \neq j(u')} \ell_{t_r} \right)$ for some $j(u') \in \{i, i+1\}$. By Lemma \ref{lemma topological ruled surfaces}, $h_{\@H}$ maps $\Sp(\@H)_{u'}$ into $\ell_{t_{j(u')}} \backslash \left(\bigcup_{r \neq j(u')} \ell_{t_r} \right) = \ell_{t_{j(u')}} \cap W'$, which is isomorphic to either $\#A^1$ or $\#A^1 \backslash \{0\}$. We will examine these cases separately. 

If $\ell_{t_{j(u')}} \cap W' \cong \#A^1$, we simply use Lemma \ref{lemma topological ruled surfaces} to deduce that $h_{\@H'}$ maps $\Sp(\@H')_{u'}$ into $\ell_{t_j}$ and hence into $W'$. This proves our claim in this case. 

Now, suppose that $\ell_{t_{j(u')}} \cap W' \cong \#A^1 \backslash \{0\}$, which is $\#A^1$-rigid.  Then $h_{\@H}$ maps $\Sp(\@H)_{u'}$ to the point $\alpha_1(u')$. In particular, $\alpha_1(u') = \alpha_2(u')$ for every such point $u'$.   By Remark \ref{remark corner lines}, $h_{\@H'}$ maps $\Sp(\@H')_{u'}$ to a single point which is, in fact, the generic point of $\ell_{t_{j(u')}}$. Thus, the $1$-ghost homotopy $\@H'$ factors through the open subscheme $W'$ of $X'$. The argument to prove Case B may now be completed as in Case A. 
\end{proof}

\section{The general case}
\label{section general case}

In this section, we prove the main results (Theorem \ref{theorem ruled iterations of S} and Theorem \ref{main theorem}) using the results proved in Sections \ref{section homotopies on a blowup}, \ref{section geometry of ruled surfaces} and \ref{section nodal blowups}.  

Let $S$ be a two dimensional scheme over $k$, and let $s \in S(k)$ be such that $S$ is essentially smooth at $s$. Let $\tilde S \to S$ be a proper birational morphism which is an isomorphism over $S \backslash \{s\}$ and suppose that $\tilde S$ is essentially smooth at all points which map to $s$. Then the morphism $\~S \to S$ is the blowup of $S$ at a closed subscheme $T$, which is supported on $s$. Moreover, if $\tilde S$ is essentially smooth over $k$, then $\tilde S \to S$ can be written as a composition $\tilde S = S_n \to S_{n-1} \to \cdots \to S_0 = S$, where $S_{i+1}$ is obtained from $S_i$ by blowing up a smooth point that lies in the preimage of $s$ with respect to the morphism $S_{i} \to S$; such a smooth point is said to be \emph{infinitely near} to $s$. The integer $n$ does not depend on the way in which the morphism $\tilde S \to S$ is written as a composition of such blowups. Indeed, it is just the number of components of the preimage of $s$ with respect to $\tilde S \to S$.  

\begin{lemma}
\label{lemma etale pullback}
Let $S$ be a two-dimensional scheme over $k$ and let $s$ be a $k$-valued point of $S$ such that $S$ is essentially smooth at $s$. Let $\tilde S$ be obtained by blowing up $n$ points infinitely near to $s$. Let $\phi: T \to S$ be an \'etale morphism and let $\tilde T = \tilde S \times_{S,\phi} T$. Let $\phi^{-1}(s) = \{t_1, \ldots, t_m\}$. Then $\~T$ is obtained from $T$ by blowing up $n$ points infinitely near to $t_i$ for each $i$, $1 \leq i \leq m$. 
\end{lemma}

\begin{proof}
Clearly, the morphism $\tilde T \to T$ is an isomorphism over $T \backslash \phi^{-1}(s)$. Also, $\~T$ is essentially smooth over $k$ and thus, the morphism $\tilde T \to T$ is the composition of blowups of finitely many points infinitely near to $t_i$. Thus, it remains to be shown that the scheme $\tilde T \times_{T,t_i} \Spec k$ has $n$ components for each $i$. We observe that
\begin{align*} 
\tilde T \times_{T,t_i} \Spec k & \simeq (\tilde S \times_{S,\phi} T) \times_{T,t_i} \Spec k\\
& \simeq \tilde S \times_{S, t_i \circ \phi} \Spec k \\
& \simeq \tilde S \times_{S, s} \Spec k,   
\end{align*}
which completes the proof. 
\end{proof}

\noindent {\bf Proof of Theorem \ref{theorem ruled iterations of S}.} 

Part (a) is already handled in Proposition \ref{proposition P1 bundle over an A1-rigid scheme}.  We now give a proof of part (b).  Let $E$ be a $\mathbb P^1$-bundle over a smooth projective curve $C$ of genus $g>0$ over an algebraically closed field $k$ of characteristic $0$. Let $X$ be a smooth projective surface, which is not isomorphic to $E$ and has $E$ for its minimal model.  Let $U = \Spec R$ be a smooth henselian local scheme over $k$. Let $\mathfrak{m}$ denote the maximal ideal of $R$ and let $u$ be the associated closed point of $U$. Since $C$ is $\A^1$-rigid, by Lemma \ref{lemma strategy ruled}, if two morphisms $U \to X$ are $n$-ghost homotopic for any $n$, their compositions with the map $X \to C$ have to be the same. Thus, we fix a morphism $\gamma: U \to C$ and characterize $n$-ghost homotopy classes of $\gamma$-morphisms $U \to X$ in the sense of Definition \ref{definition homotopy respecting fibres}. Let $c_0 = \gamma(u)$. Note that all the $n$-ghost homotopies of $\gamma$-morphisms $U \to X$ factor through the pullback of $X$ by the morphism $\Spec \@O_{C,c_0}^h \to C$.  The pullback of $E$ by the morphism $\Spec \@O_{C,c_0}^h \to C$ is just $\P^1 \times {\Spec \@O_{C,c_0}^h}$.  We may therefore replace $C$ with $\Spec \@O_{C,c_0}^h$.  If $x$ is a local parameter at $c_0$, then $\@O_{C,c_0}^h$ is isomorphic to the Henselization of $k[x]$ at the ideal $\<x\>$. Thus, we will simply take $C$ to be as in Notation \ref{notation nodal blowups} after which $E$ is seen to be isomorphic to $\#P^1_C$. For any scheme $S$ over $C$, we will denote the fiber $S \times_{C,c_0} \Spec \kappa(c_0)$ of $S \to C$ over $c_0$ by $S_{c_0}$.

The birational morphism $X \to \#P^1_C$ can be viewed as the blowup of $\#P^1_C$ at the ideal sheaf $\@J$, the support of which, denoted by $\Supp(\@J)$, is a closed subscheme of codimension $2$. For any point $Q \in \Supp(\@J)$, let $N_Q$ denote the number of components of the fiber of $X \to \#P^1_C$. We know that the morphism $X \to \#P^1_C$ can be written as a composition of a sequence of blowups at smooth closed points. The number $N_Q$ can also be interpreted as the number of points infinitely near to $Q$ which have been blown up in this manner. We will prove the result by induction on $$N_X:= \max_{Q \in \Supp(\@J)} N_Q.$$ The case $N_X=1$ is clearly handled by the results of Section \ref{section nodal blowups}.

As we observed in Remark \ref{remark trivial case}, we are reduced to considering the case in which $\gamma$ maps the closed point $u$ of $U$ to the closed point $c_0$ of $C$ and the generic point of $U$ to the generic point of $C$. We write $r_0 = \gamma^*(x) \in R$ for the pullback of the uniformizing parameter $x$ on the coordinate ring of $C$. 

We use the notation from the proof of Theorem \ref{theorem nodal case}, which we quickly recall for the sake of convenience. Any $\gamma$-morphism $\alpha: U \to \#P^1_C$, along with the identity map $U \to U$ induces a morphism $\beta: U \to \#P^1_U= U \times \#P^1 := U \times_{\gamma, C} \#P^1_C$. If the $\gamma$-morphism $\alpha$ maps the closed point $u$ of $U$ to any point of the form $(c_0, [p:1])$ for $p \in \kappa(c_0)$, then the corresponding section $\beta: U \to \#P^1_U$ maps $u$ to the point $(c_0, [p:1])$ where $p$ is now seen as an element of the field $R/\mathfrak{m}$. Hence it factors through the open subscheme $\#A^1_U = \Spec R[y] = \Spec R[Y/Z] \subset \#P^1_U$. Thus, any such $\beta$ is characterized by an $R$-algebra homomorphism $\beta^*: R[y] \to R$. Clearly, $\beta^*(y) = \alpha^*(y)$. Moreover, $\alpha$ maps $u$ to $(c_0, [0:1])$ if and only if $\alpha^*(y)$ lies in the maximal ideal $\mathfrak{m}$ of $R$.  For any $r \in R$, let $\beta_r: U \to \#P^1_U$ be the morphism induced by the $R$-algebra homomorphism $R[y] \to R$, $y \mapsto r$. Note that any section $\beta: U \to \#P^1_U$ such that $\beta(u) = (u, [p:1])$ for some $p \in R/\mathfrak{m}$ is of the form $\beta_r$ for some unique $r \in R$.   

We now rephrase our problem as follows:  With $C$, $X$, $U$ and $\gamma: U \to C$ as above, let $\alpha, \alpha': U \to \#P^1_C$ be two $\gamma$-morphisms. Suppose that $\alpha, \alpha'$ are connected by an $n$-ghost homotopy $\@H$ which lifts to $X$ (which implies that $\alpha$ and $\alpha'$ also admit lifts to $X$). We would like to prove that $\alpha$ and $\alpha'$ are connected by a $1$-ghost homotopy which satisfies the hypothesis of Lemma \ref{lemma S^2 agrees with pi0} and also lifts to $X$.  In view of Proposition \ref{proposition local generator ghost homotopy}, we have the following two cases.

\noindent \emph{\bf Case I:}  $\alpha$ and $\alpha'$ map $u$ into $\#P^1_C \setminus \Supp(\@J)$. 

If $\Supp(\@J)$ is a singleton, then it is easy to see that $\alpha$ and $\alpha'$ are $\A^1$-chain homotopic.  Now suppose that $\Supp(\@J)$ consists of at least two closed points.  Without any loss of generality, we may assume that two of these points are $(c_0, [0:1])$ and $(c_0,[1:0])$. Thus, as we saw above, there exist $r,r' \in R$ such that $\beta_{r}, \beta_{r'}: U \to \#P^1_U$ are the pullbacks of $\alpha$, $\alpha'$ respectively, where $r, r'\in R$. Let $\~{\@H}$ be the pullback of $\@H$ to $\#P^1_U$.  Since $\@H$ lifts to $X$, $\~{\@H}$ lifts to $X_{\gamma}$. Thus, by Proposition \ref{proposition local generator ghost homotopy}, $\~{\@H}$ factors through the complement of the preimage of $\Supp(\@J)$ with respect to the morphism $\#P^1_U \to \#P^1_C$. Thus $\@H$ factors through the complement of $\Supp(\@J)$. Hence, the restriction of $h_{\@H}$ to $h_{\@H}^{-1}(\#P^1 \times \{c_0\})_{\rm red}$ factors through $\left(\#P^1 \backslash \{ [0:1], [1:0] \} \right) \times \{c_0\}$ which is $\#A^1$-rigid. Thus, the restriction of $h_{\@H}: \Sp(\@H) \to \#P^1_C$ to $h_{\@H}^{-1}(\#P^1 \times \{c_0\})_{\rm red}$ is constant. Therefore, for every prime ideal $\mathfrak{p}$ of $R$ containing $r_0 = \gamma^*(x)$, we have $r-r' \in \mathfrak{p}$.  Hence, it follows that $r-r' \in \sqrt{\<r_0\>}$.  Now consider the homotopy $\#A^1_U = \Spec(R[T]) \to \Spec(R[Y/Z]) = \#A^1_U \subset \#P^1_U$, defined by the ring homomorphism $R[Y/Z] \to R[T]$, $Y/Z \mapsto r(1-T) + r'T$, which connects $\beta_{r}, \beta_{r'}: U \to \#P^1_U$. Since $r-r' \in \sqrt{\<r_0\>}$, one easily sees that this homotopy avoids $\Supp(\@J)$ and consequently, lifts to $X$. This proves the result in Case I (observe that this argument proves the result for all values of $N_X$). 

\noindent \emph{\bf Case II:} $\alpha, \alpha'$ map $u$ into $\Supp(\@J)$.  By the results of Section \ref{subsection homotopies surfaces}, $\alpha, \alpha'$ must map $u$ to the same point $P \in \Supp(\@J)$.

We use Lemma \ref{lemma single point support} to obtain $X' \in \@N$ such that the morphism $X \to \#P^1_C$ factors as
\[
X \stackrel{\phi}{\to} X' \stackrel{\psi}{\to} \#P^1_C,
\]
where $\phi$ and $\psi$ are compositions of successive blowups at smooth closed points and $\phi:X \to X'$ is the blowup of an ideal sheaf $\@J'$, the support of which does not contain any node of $X'$.

Let us denote the lifts of $\alpha$, $\alpha'$ and $\@H$ to $X'$ by $\~{\alpha}$, $\~{\alpha}'$ and $\~{\@H}$ respectively. It will suffice to show that $\~{\alpha}$ and $\~{\alpha}'$ are connected by a $1$-ghost homotopy $\@H'$ of $U$ on $X'$ which satisfies the hypothesis of Lemma \ref{lemma S^2 agrees with pi0} and also lifts to $X$.

\noindent \emph{\bf Case II-A:} $\~{\alpha}(u)$ lies in the intersection of two distinct lines of $X'$.  

For $1 \leq i \leq p$, let $P_i$ be the point of intersection of $\ell_{i-1}$ and $\ell_i$. Suppose that $\~{\alpha}(u) = P_{i_0}$.  The proof of Theorem \ref{theorem nodal case} shows that $\~{\alpha}$ and $\~{\alpha}'$ are connected by a $1$-ghost homotopy $\@H'$ which satisfies the hypothesis of Lemma \ref{lemma S^2 agrees with pi0}.  We would like to choose $\@H'$ in such a way that for any $u' \in U$ for which $\gamma(u') = c_0$, the morphism $h_{\@H'}$ maps $\Sp(\@H')_{u'}$ into the complement of $\Supp(\@J')$. This will be true for any $u'$ such that $\~{\alpha}(u') = P_{i_0}$ (in particular, this includes the case $u' = u$) since Lemma \ref{lemma topological ruled surfaces} implies that, in this case, $\Sp(\~{\@H})_{u'}$ is mapped to $P_{i_0}$, which is not in $\Supp(\@J')$.  The rest of the argument is along the same lines as the argument in Case B in the proof of Proposition \ref{proposition nodal reduction}. 

Suppose $u' \in U$ is such that $\gamma(u') = c_0$, but $\~{\alpha}(u') \neq P_{i_0}$. Since $\gamma(u') = c_0$, $\~{\alpha}(u')$ lies in the closed fiber $X'_{c_0} = \bigcup_{i=0}^p \ell_{s_i}$. Since $\~{\alpha}(u)$ is in the closure of $\~{\alpha}(u')$, it follows that $\~{\alpha}(u')$ lies in $\ell_{s_{i_0-1}} \cup \ell_{s_{i_0}}$. As $\~{\alpha}(u') \neq P_{i_0}$, $\~{\alpha}(u')$ lies in $\ell_{s_{i_0-1}} \backslash \ell_{s_{i_0}}$ or $\ell_{s_{i_0}} \backslash \ell_{s_{i_0 - 1}}$. Let $j(u') \in \{i_0 - 1, i_0\}$ be such that $\~{\alpha}(u') \in \ell_{s_{j(u')}}$. Since $\~{\alpha}(u) = P_{i_0}$ lies in the closure of $\~{\alpha}(u')$ but is not equal to $\~{\alpha}(u')$, it follows that $\~{\alpha}(u')$ is equal to the generic point of $\ell_{s_{j(u')}}$. Since $\~{\alpha}(u')$ lies only on the component $\ell_{s_{j(u')}}$ of $X'_{c_0}$, it follows by Lemma \ref{lemma topological ruled surfaces} that for any $1$-ghost homotopy $\@H'$ of $U$ on $X'$ connecting $\~{\alpha}$ and $\~{\alpha}'$, $\Sp(\@H')_{u'}$ is mapped into the open subscheme $X' \backslash \left(\bigcup_{i \neq j(u')} \ell_{s_i} \right)$ of $X'$. Thus, $h_{\@H'}$ maps $\Sp(\@H')_{u'}$ into $\ell_{s_j(u')} \backslash \left(\bigcup_{i \neq j(u')} \ell_{s_i} \right)$.  

If $\Supp(\@J') \cap \ell_{j(u')} = \emptyset$, it is obvious that for a $1$-ghost homotopy $\@H'$ as above, $h_{\@H'}$ maps $\Sp(\@H')_{u'}$ into the complement of $\Supp(\@J')$.  If $j(u') \in \{0,p\}$, then using Remark \ref{remark corner lines}, we see that $\@H'$ lifts to $X$.  If $j(u') \notin \{0,p\}$, then $\ell_{j(u')}$ meets two components of $X'_{c_0}$ and so $\ell_{j(u')} \backslash \left(\bigcup_{i \neq j(u')} \ell_i \right)$ is isomorphic to $\#A^1 \backslash \{0\}$, which is $\#A^1$-rigid. Thus, $\Sp(\~{\@H})_{u'}$ is mapped to the generic point of $\ell_{s_{j(u')}}$. In particular, it factors through the complement of $\Supp(\@J')$. This proves the result in Case II-A. 

\noindent \emph{\bf Case II-B:} $\~{\alpha}(u)$ lies on a single line of $X'$. 

Let $r,r' \in R$ be such that $\beta_r, \beta_{r'}: U \to \#P^1_U$ are the pullbacks along $\gamma$ of $\alpha, \alpha'$ respectively. Suppose $\~{\alpha}(u)$ lies only on the line $\ell_{s}$ of $X'$, where $s \in \{s_0, \ldots, s_p\}$. Then the homotopy $\~{\@H}$ factors through $X' \backslash \left(\bigcup_{s_i \neq s} \ell_{s_i} \right)$.

There exists a unique rational number $s$ and a line labelled $\ell_s$ in $X'$ such that $\~P$ lies on $\ell_s$.  Let $a,b$ be coprime positive integers such that $s=a/b$ is the reduced expression.   By Proposition \ref{proposition nodal reduction}, we may assume that $s = a/b < 1$.  Then $x^a/y^b$ is a parameter on the line $\ell_s$.  It follows that $r, r'$ are unit multiples of each other in $R$ and that $r_0^a/r^b \in R^{\times}$.  Since $R$ is henselian local over $k$ and $k$ is algebraically closed, $r_0^a/r^b$ has a $b$-th root in $R$.  Consequently, $r_0^a$ has a $b$-th root in $R$ and since $a$ and $b$ are coprime, we conclude that $r_0$ admits a $b$-th root in $R$.  Fix $\~r_0 \in R$ such that ${\~r_0}^b = r_0$.  Let $\phi_b: C \to C$ be the morphism corresponding to the $k$-homomorphism $k[x]_{\<x\>}^h \to k[x]_{\<x\>}^h$ defined by $x \mapsto x^b$.  By Theorem \ref{theorem etale cover}, we have a $b$-sheeted \'etale cover (see Definition \ref{definition U_r} for the definition of the open subscheme $U_s$ of $X'$)
\[
\~X'= (U_s \times_{C, \phi_b} C)[x^a/y]  \to (U_s \times_{C, \phi_b} C) \to U_s
\]
and an isomorphism $U_{a} \xrightarrow{\sim} \~X'$ of $C$-schemes.  By Proposition \ref{proposition local generator ghost homotopy}, we see that $\@H$ factors through the open subscheme $U_s$ of $X'$.  Using the fixed $b$-th root $\~r_0$ of $r_0$, we see that the sections $\alpha, \alpha'$ and the $n$-ghost homotopy $\@H$ lift to $U_s \times_{C, \phi_b} C$.  We claim that $\@H$ lifts to $\~X'$.  
\[
\xymatrix{
 & \~X' \ar[d] \\
 & U_s \times_{C, \phi_b} C \ar[d] \\
\Sp(\@H) \ar[r]^-{h_{\@H}}\ar[ur] \ar[d]_-{f_{\@H}} & U_s \ar[d] \\
U \ar[r]^-{\alpha, \alpha'}        & X'
}
\]
Now, the rational function $x^a/y^b$ defines a morphism $X' \to \#P^1_C$ (Lemma \ref{lemma zeros and poles}) and its pullback $x^{ab}/y^b = \left( x^a/y \right)^b$ to $U_s \times_{C, \phi_b} C$ defines a morphism $\Sp(\@H) \to \#P^1_C$.  Since $\Sp(\@H)$ is normal, the pullback of the rational function $x^a/y$ to $\Sp(\@H)$ is regular at every point of $\Sp(\@H)$.  By Proposition \ref{proposition resolving indeterminacies}, $h_{\@H}$ lifts to $\~X' = (U_s \times_{C, \phi_b} C)[x^a/y] \simeq U_a \simeq U_0$, where the last isomorphism is given by Lemma \ref{lemma elementary transformation}.  Set $\~X : = \~X' \times_{X'} X$, where the morphism $\~X' \to X'$ is given by the composition $\~X' \to U_s \hookrightarrow X'$, which is \'etale.  Now, the lifts of $\alpha$ and $\alpha'$ to $X$ are $n$-ghost homotopic if and only if their lifts in $\~X$ are $n$-ghost homotopic.  Since $m'<m$, it can be seen from Lemma \ref{lemma etale pullback} that there exists a smooth, projective ruled surface $X''$ such that 
\begin{itemize}
\item $X''$ is obtained by successive blowups of $\P^1_C$ at finitely many smooth, closed points;
\item $\~X'$ is an open subscheme of $X''$;
\item $\alpha, \alpha'$ lift to $X''$ and their lifts in $X''$ are $n$-ghost homotopic if and only if the lifts of $\alpha, \alpha'$ in $\~X'$ are $n$-ghost homotopic;  and 
\item $N_{X''}<N_X$.
\end{itemize}

By the induction hypothesis, we conclude that the lifts of $\alpha$ and $\alpha'$ to $X''$ are $1$-ghost homotopic and map to the same element in $\pi_0^{\A^1}(X'')(U)$.  Moreover, the same holds for the lifts of $\alpha$ and $\alpha'$ to $\~X'$ and to $X$.  This completes the proof of Theorem \ref{theorem ruled iterations of S}.   \qed

\noindent {\bf Proof of Theorem \ref{main theorem}.} 

Let $X$ be a smooth, projective birationally ruled surface over an algebraically closed field $k$.  By \cite[Chapter III, Theorem 2.3]{Kollar}, $X$ can be obtained by successively blowing up a finite number of smooth, closed points on $\#P^2$ or a $\P^1$-bundle over a curve $C$.  If $X$ is rational, then by \cite[Corollary 2.3.7]{Asok-Morel}, $X$ is $\A^1$-connected.  Hence, $\pi_0^{\#A^1}(X)$ is evidently $\A^1$-invariant, being the trivial one-point sheaf.  If $X$ is not rational, then $X$ has to be birationally ruled over a curve $C$ of genus $>0$, in which case the result is proved in Theorem \ref{theorem ruled iterations of S} above. \qed

\begin{remark}
\label{remark naive ruled} 
It has been shown in \cite[Theorem 1]{Balwe-Sawant-ruled2} that for a smooth projective surface $X$ birationally ruled over a curve $C$ of genus $>0$ over an algebraically closed field $k$, one has $$\pi_0^{\A^1}(X)(U) \simeq \mathcal S^n(X)(U),$$ for all $n \geq 1$ and for all $U \in Sm/k$ of dimension $\leq 1$.  However, if $X$ is not a minimal model, then \cite[Theorem 2]{Balwe-Sawant-ruled2} shows that $\mathcal{S}(X) \nsimeq C$ and that $\mathcal S(X) \neq \mathcal S^2(X)$.  This shows, in particular, that the Morel-Voevodsky singular construction $\Sing X$ is not $\A^1$-local for such $X$. 
\end{remark}

\begin{remark}
\label{remark S future} 
Given a smooth, projective variety $X$ over a field $k$, it is an interesting question to determine if there exists a natural number $n_X$ (possibly depending only on the dimension of $X$), where $n_X$ is the least with the property that $$\@S^{n_X}(X) \simeq \@S^{n}(X) \simeq  \pi_0^{\A^1}(X),$$ for all $n \geq n_X$.  If $X$ is a smooth projective surface over an algebraically closed field of characteristic $0$, then Theorem \ref{theorem ruled iterations of S} shows that $n_X \leq 2$.  It seems reasonable to conjecture that $n_X \leq \dim X$, for a general smooth projective variety $X$ over an algebraically closed field $k$. 
\end{remark}

\end{document}